\documentclass[11pt,reqno]{amsart}
\usepackage{latexsym,amsmath,amssymb,mathrsfs,xcolor,amsthm}
\usepackage{enumerate}
\usepackage{graphicx}
\usepackage{units}
\usepackage{ esint }
\usepackage{comment}
\usepackage{url}

\numberwithin{equation}{section}

plus 6pt minus 12pt
\newtheorem{lemma}{Lemma}[section]

\renewcommand{\d}{\,{\rm d}}

\newtheorem{theorem}{Theorem}

\title{Sharp Transitions for Localized Solutions to a Diophantine Inequality } 
\author{Ataleshvara Bhargava}
\subjclass[2020]{11D75, 11P05, 11P55}

\begin{document}

\begin{abstract} For fixed $\tau > 0$, non-integer $\theta > 2$ and large enough $s$, we investigate the number of solutions to the Diophantine inequality $|x_1^{\theta}+\cdots +x_s^{\theta} - R| < \tau$ as $R \to \infty$. Here, we restrict the variables $x_i$ in the ``almost diagonal" range $X-Y < x_i \leq X+Y$ for $i = 1, \ldots, s$, where $X = (R/s)^{1/\theta}$ and $Y \asymp \sqrt{X}$. Let $\omega = (\lfloor s/2 \rfloor(\theta-1))^{-1/2}$. We will show that if $Y = c\sqrt{X}$ for some $c > \omega$ then for sufficiently large $R$ there must always exist solutions, but if $c < \omega$ then there exist arbitrarily large positive $R$ for which there are no solutions. Our result is thus essentially sharp, with the exception of $c = \omega$. This work is analogous to the results of Daemen \cite{DaemenAA1} and Wright \cite{Wright1} in which similar statements are proved for Waring's problem, though our results are likely somewhat stronger than what is possible in that setting. We discuss other related results and further work to be done. 

\end{abstract}

\maketitle 

\everymath{\displaystyle}

\section{Introduction} 
Our objective in this article is to study surprising properties of a certain subclass of integer solutions to a real-exponent variant of Waring's problem. Fix $\tau > 0$ and $\theta > 2$ with $\theta \notin \mathbb{N}$, and let $s \in \mathbb{N}$.  
For positive $R \in \mathbb{R}$, define $X = (R/s)^{1/\theta}$ and let $Y \in \mathbb{R}$ with $0 \leq Y \leq X$. Let $N_{s,\theta}^{\tau}(R,Y)$ be the number of positive integer solutions $(x_1, \ldots, x_s)$ to the relation 
\begin{align} \label{eq:main}
|x_1^{\theta} + \cdots +x_s^{\theta} - R| < \tau, \quad X-Y < x_i \leq X+Y \quad \text{ for each } \ i = 1, \ldots, s.
\end{align}
Essentially, we want to count the number of ways to approximate $R$ as a sum of $s$ real numbers that are each $\theta$th powers of positive integers in a restricted range as $R \to \infty$. Setting $x_i = X$ for all $i$ would constitute the diagonal solution to \eqref{eq:main}, but this is not possible in general since $X$ may not be an integer. However, the restriction $X-Y < x_i \leq X+Y$ forces the variables $x_i$ to be close to the diagonal $X$ and to one another in $\ell^{\infty}$-norm, for which reason we call these solutions ``almost-diagonal" solutions or ``localized" solutions. We will specifically be interested in the case when $Y = c\sqrt{X}$ for some constant $c = c(s,\theta,\tau)$. It turns out that the behavior of $N_{s,\theta}^{\tau}(R, Y)$ is heavily dependent on the value of $c$. If $s$ and $c$ are large enough, then an asymptotic lower bound holds for $N_{s,\theta}^{\tau}(R, Y)$, which we record as follows.

\begin{theorem} \label{thm:pos}
Let $\theta > 2$ be a non-integer and fix $\tau > 0$. For a positive real number $R$, let $X = (R/s)^{1/\theta}$ and $Y=c\sqrt{X}$, where $c >(\lfloor s/2 \rfloor (\theta-1))^{-1/2}$ is a fixed constant. Then for any positive integer $s \geq (\lfloor2\theta\rfloor+1)(\lfloor 2\theta \rfloor+2)+1$, the bound 
\[ N_{s, \theta}^{\tau}(R,Y) \gg_{s,\theta} \tau X^{1-\theta}Y^{s-1} \]
holds for all sufficiently large real $R > 0 $. 
\end{theorem}

An analogous result was proven in the context of Waring's problem by Daemen \cite{DaemenAA1}, where $\theta$ and $R$ are replaced with positive integers no smaller than $2$ and approximation is replaced with equality. Our proof of Theorem \ref{thm:pos} will make use of the Davenport-Heilbronn variant of the circle method. 

If $c$ is too small in terms of $s$ and $\theta$, then the above bound may not hold, as our next result shows. A similar result was proven in the case of Waring's problem by Wright \cite{Wright1}. For a sequence of real numbers $0 \leq R_1 < R_2< \ldots$, define $X_n = (R_n/s)^{1/\theta}$. 

\begin{theorem} \label{thm:neg} 
Let $\theta > 2$ be a non-integer and fix $\tau > 0$. Fix a constant $c < (\lfloor s/2 \rfloor(\theta-1))^{-1/2}$ and a positive integer $s \geq 2$. There exists an infinite sequence of real numbers $\{R_n\}_{n=1}^{\infty}$ with $R_n \to \infty$ as $n \to \infty$ for which  $N_{s,\theta}^{\tau}(R_n, c\sqrt{X_n}) = 0$. 
\end{theorem}

It is useful to provide motivation for why one would consider this problem. One of the most natural ways one might try to construct solutions to \eqref{eq:main} would be to try the diagonal solution $x_1 = \cdots = x_s = X$, but this obviously fails if $X$ is not an integer. The next natural progression would be to consider integer values close to $X$, say $X-Y < x_i \leq X+Y$ for some $Y$. We want to see how small we can take $Y$ and still prove the existence of solutions to \eqref{eq:main}. 
Theorems \ref{thm:pos} and \ref{thm:neg} show that the optimal size of $Y$ is $Y = c \sqrt{X}$, and $N_{s,\theta}^{\tau}(R,Y)$ displays a sudden change in behavior at this scale depending on $c$.

We pause to discuss some interesting features of these two results, which lead to further questions. The first question inherent to most counting problems of this type concerns the size of $s$ in terms of $\theta$. In Theorem \ref{thm:pos} we require $s \geq (4+o(1))\theta^2$, which matches the improvements obtained by Poulias \cite{Poulias1} for Diophantine inequalities with fractional exponents. Both the result of this paper and of \cite{Poulias1} make use of the arguments of Arkhipov and Zhitkov \cite{ArkhZhit} and the powerful estimates now available as consequences of the now resolved Main Conjecture in Vinogradov's Mean Value Theorem; see Bourgain, Demeter and Guth \cite{BDG} and Wooley \cite{WooNEC}. The bounds from \cite{DaemenAA1} are essentially $s \geq (10/3+o(1))k^2\log(k)$, where $k$ plays the analogous role to $\theta$. Using the results of \cite{BDG} and \cite{WooNEC} one could likely improve this to $s \geq k^2+k+1$. Upon inspection of the exponent in the lower bound to $N_{2,\theta}^{\tau}(R,Y)$ in Theorem \ref{thm:pos}, one notices that the exponent is positive as soon as $s > 2\theta-1$, and so it would be natural to conjecture that Theorem \ref{thm:pos} holds as soon as $s > 2 \theta-1$. However, the analogous barrier in the standard Waring problem and its diagonal and real-exponent variants (barring any local restrictions) is $s > k$ (or $s > \theta$). Thus, one is left to wonder what the behavior of almost-diagonal solutions is when $\theta < s \leq 2\theta -1$. If almost-diagonal solutions do exist, their asymptotics may be different, or the value of $Y$ may have to be different, say $cX^{\kappa}$ for some $\kappa > 1/2$ depending on $s$ or $\theta$. One may formulate a natural conjecture for what $\kappa$ should be when $\theta < s \leq 2\theta-1$, but for obvious reasons we are very far from proving any results in this direction anyway. Note also that we only require $s \geq 2$ in Theorem \ref{thm:neg}. 

One may also consider the asymptotics for $N_{s,\theta}^{\tau}(R,Y)$. Wright \cite{WrightAsymp} proved an asymptotic formula for the analogous quantity to $N_{s,\theta}^{\tau}(R,Y)$ in the context of Waring's problem, but with the more restrictive condition that $Y \geq X^{1-1/k+\epsilon}$ for some $\epsilon > 0$, where $k$ is the degree. This was improved by Daemen \cite{DaemenBLMS} to the condition $Y \geq \log(X)^{r_k}\sqrt{X}$ where $r_k = (10/3 +o(1))k^2\log(k)$. 
However, an asymptotic formula for when $Y = c\sqrt{X}$ still seems to be out of reach. 
The same question remains open in the non-integer degree setting. 

It is not inconceivable that this is tractable even when the integer degree case is not; our conditions on $c$ for existence are sharper than those in \cite{DaemenAA1} due to better control from the lack of singular series (Daemen \cite{DaemenAA1} requires $c \geq 4$, though this can likely be reduced with some care). 
It would also be interesting to consider asymptotics for $N_{s,\theta}^{\tau}(R,Y)$ with $Y$ having other orders of growth in terms of $X$ than $Y = c\sqrt{X}$. 

One of the most striking features of this result is that a mere constant factor can drastically change the behavior of $N_{s,\theta}^{\tau}(R,Y)$. Moreover, excepting the critical point $\omega = (\lfloor s/2 \rfloor(\theta-1))^{-1/2}$ we establish essentially sharp conditions on $Y$ for the existence of solutions. It then becomes a particularly interesting problem to determine the behavior of $N_{s,\theta}^{\tau}(R,\omega \sqrt{X})$, as this critical situation is not covered here. It might be the case that there are constants $\delta < 1/2$ and $\omega_0$ such that $N_{s,\theta}^{\tau}(R,\omega\sqrt{X}+cX^{\delta})$ is positive for all large $R$ if $c > \omega_0$ but not if $c < \omega_0$.

There are also questions in regards to Theorem \ref{thm:neg}. We show the existence of a particular sequence which cannot be approximated by almost diagonal solutions, but one can ask about the density or frequency of such real numbers, and an exceptional set result along these lines could be interesting. The strongest result of this type would determine the density or frequency of the set of real numbers $R$ for which $N_{s,\theta}^{\tau}(R,Y) = 0$, and how this depends on the magnitude of $Y$, whether $Y = c\sqrt{X}$ for a constant $c$ or whether $Y \asymp X^{\delta}$ for some $\delta < 1/2$. One may also determine how these density and exceptional set results depend on the other parameters, such as $\theta$ and $s$ or even $\tau$. 

We finally mention some results related to Theorems \ref{thm:pos} and \ref{thm:neg} which exist for other Diophantine problems. The work of Daemen \cite{DaemenAA1} and Wright \cite{Wright1} on Waring's problem remain the main sources of inspiration for this work. Similar results also exist for the Waring-Goldbach problem with exponent $k$. Wei and Wooley \cite{WW} show the existence of almost-diagonal solutions as long as $Y \geq X^{5/6+\epsilon}$ and $s > 2k(k-1)$ if $k \geq 4$, and prove several other related results. Salmensuu \cite{Sal} improved this by showing that for large enough $s$ in terms of $k \geq 2$ one can take $Y \geq X^{0.525+\epsilon}$, and for $s \geq k^2+k+1$ with $k \geq 4$ one can take $Y \geq X^{0.55+\epsilon}$. See \cite{Sal} and works cited therein for other related results. Biggs \cite{Biggs} proves a result similar to Theorem \ref{thm:neg} in the case of Waring's problem with shifts. The result of \cite{Biggs} goes even further to show that approximations may not even exist if $\tau$ is allowed to grow with $R$, up to the level $\tau < \mu R^{1-2/k}$ for some constant $\mu$, where $k$ is the degree. Following this, it may also be interesting to pursue similar results in this context, both in the affirmative and in the negative, with $\tau$ being allowed to grow with $R$; one may even be able to prove existence of solutions with a smaller number of variables or with a smaller value of $Y$ if $\tau$ grows quickly enough with $R$. 

As a quick word on notation, we make free use of Vinogradov's notation $P \ll Q$, and use $P \asymp Q$ to mean $P \ll Q$ and $P \gg Q$. For the rest of this paper, we assume $R$ is a large positive real number, that $X = (R/s)^{1/\theta}$, and that $Y = c\sqrt{X}$ for some constant $c$ whose value will be specified as needed. We also fix $\tau > 0$ and $\theta > 2$ with $\theta \notin \mathbb{N}$. We write $\mathbb{N} = \{1,2,\ldots,\}$ to be the set of positive integers. As usual, we adopt the notation $e(x) = e^{2\pi i x}$ for $x \in \mathbb{R}$. 
The rest of this paper is organized as follows. In Section \ref{sec:neg} we prove Theorem \ref{thm:neg}. In Section \ref{sec:prelim} we outline the method of proof for Theorem \ref{thm:pos} and make the initial set-up for our application of the Davenport-Heilbronn method. In Section \ref{sec:min} we prove bounds for the minor arcs, and in Section \ref{sec:triv} we dispose of the trivial arcs. Finally, we conclude in Section \ref{sec:maj} with the major arc analysis.

\subsection{Acknowledgments} The author would like to thank his advisor, Prof. T.D. Wooley, for funding support from NSF grant DMS-2502625 and for many helpful discussions and suggestions. The author also thanks the anonymous referee for helpful comments improving the exposition of this paper. 

\section{Proof of Theorem \ref{thm:neg}} \label{sec:neg}

We begin by proving Theorem \ref{thm:neg}, that if $Y = c\sqrt{X}$ and $c$ is too small then one cannot guarantee the existence of solutions to \eqref{eq:main}. Our method is an adaptation of the proof in \cite{Wright1}. This is a direct consequence of the following. 

\begin{lemma} Fix $\tau > 0$ and $\theta > 2$ with $\theta \notin \mathbb{N}$, and let $s \in \mathbb{N}$ with $s \geq 2$. Let $M \in \mathbb{N}$ be sufficiently large. Fix $c < (\lfloor s/2 \rfloor (\theta-1))^{-1/2} $. Then there are no solutions to \eqref{eq:main} when $R = sM^{\theta}+s\theta M^{\theta-1}$. 
\end{lemma} 

\begin{proof} Let $R = sM^{\theta}+s\theta M^{\theta-1}$, where $M$ is a large integer. Recall that $X = (R/s)^{1/\theta}$. Assume for sake of contradiction that there are $x_1,\ldots, x_s$ such that \eqref{eq:main} is true, with $Y = c\sqrt{X}$. Let $A_i = x_i-M$ for each $i$. Note that since $X^{\theta}= R/s = M^{\theta}+\theta M^{\theta-1}$, we have $X \geq M$. We may therefore deduce from the Mean Value Theorem that 
\[ \theta M^{\theta-1}|X-M| \leq |X^{\theta}-M^{\theta}| = \theta M^{\theta-1}. \]
Therefore we have $M \leq X \leq M+1$, and also 
\begin{align*} 
|A_i| & = |x_i-M| \leq |x_i-X|+|X-M| \leq  cX^{1/2} +1 \leq c(M+1)^{1/2}+1.
\end{align*}
In particular, for any fixed $\delta > 0$ to be chosen later we have $|A_i| \leq c(1+\delta)M^{1/2}$ for all large $M$. 
One more observation about $A_i$ is that we cannot have $x_i = M$ for all $i$ and for large $M$. This would mean $x_1^{\theta}+\cdots +x_s^{\theta} = sM^{\theta}$. However, $sM^{\theta}$ is not within $\tau$ of $R$, and so this contradicts the assumption that $(x_1,\ldots,x_s)$ is a solution to \eqref{eq:main}. Therefore, we must have 
\begin{align} \label{eq:A_i}
\sum_{i=1}^s A_i^2 = \sum_{i=1}^s |x_i-M|^2 \geq 1. 
\end{align}
Using the Taylor series for the function $h(x) = (M+x)^{\theta}$ centered at the origin with number of terms $k$ and remainder $r_k$ we then have 
\begin{equation} \label{eq:comp}
\begin{aligned}
\tau & > \Big|\sum_{i=1}^s (M+A_i)^{\theta}-sM^{\theta}-s\theta M^{\theta-1}\Big| = \Big|\sum_{i=1}^s \sum_{j=0}^k \binom{\theta}{j}A_i^{j}M^{\theta-j}+\sum_{i=1}^s r_k(A_i)-sM^{\theta}-s\theta M^{\theta-1}\Big| 
\\ & = \Big|\theta M^{\theta-1}\big(-s+\sum_{i=1}^s A_i\big)+\sum_{i=1}^s \sum_{j=2}^k \binom{\theta}{j}A_i^jM^{\theta-j}+\sum_{i=1}^s r_k(A_i)\Big| 
\\ & \geq \Big|\theta M^{\theta-1}\big(-s+\sum_{i=1}^s A_i\big)\Big|-\Big|\sum_{i=1}^s \sum_{j=2}^k \binom{\theta}{j}A_i^jM^{\theta-j}+\sum_{i=1}^s r_k(A_i)\Big|. 
\end{aligned}
\end{equation}
Now, the remainder is given as 
\[r_k(A_i) = \binom{\theta}{k+1}(M+z)^{\theta-k-1}A_i^{k+1}. \] Since $|A_i| \leq c(1+\delta)M^{1/2}$, one has $M+A_i = M +O(M^{1/2})$, and so 
\[ |r_k(A_i)| \ll_{\theta,k,c,\delta} M^{\theta-k-1}M^{(k+1)/2} \ll M^{\theta-(k+1)/2}. \] 
When $k \geq 2$ it follows that 
\[ \Big|-s+\sum_{i=1}^s A_i\Big| \leq \theta^{-1}\Big|\sum_{i=1}^s \sum_{j=2}^k \binom{\theta}{j}A_i^jM^{1-j}\Big|+O_{\tau, \theta, k, c, \delta}(M^{-1/2}) \]
\[ \leq \theta^{-1}\sum_{i=1}^s \sum_{j=2}^k \Big|\binom{\theta}{j}\Big|(c(1+\delta)M^{1/2})^jM^{1-j}+o(1) \leq s\theta^{-1} \sum_{j=2}^k \Big| \binom{\theta}{j} \Big| (c(1+\delta))^j M^{1-j/2}+o(1). \]
Noting that $M^{1-j/2} = O(M^{-1/2})$ for all $j > 2$, for any $\delta > 0$ we have 
\begin{align} \label{eq:sA_i}
\Big|s-\sum_{i=1}^s A_i\Big| \leq \frac{1}{2}sc^2(1+\delta)^2(\theta-1)+o_{\tau, \theta, k, c, \delta}(1). 
\end{align}
We now consider two cases, according to whether $s$ is even or odd. Suppose first that $s$ is even. Let $k = 4$. 
Since $c < \sqrt{2}(s(\theta-1))^{-1/2}$ we may find $\delta > 0$ such that $2^{-1}s(c(1+\delta))^2(\theta-1) < 1$. Therefore, for all large $M$ we see that the inequality 
\[ \Big|-s+\sum_{i=1}^s A_i\Big| < 1 \]
holds. Since the left hand side is an integer, we deduce that this integer must be $0$. Substituting this into the second line of \eqref{eq:comp} and using $k = 4$, one has 
\[ \tau > \Big|\sum_{i=1}^s \sum_{j=2}^4 \binom{\theta}{j} A_i^jM^{\theta-j} +\sum_{i=1}^s r_4(A_i)\Big| \geq \Big|\sum_{i=1}^s \binom{\theta}{2}M^{\theta-2}A_i^2\Big| - \Big|\sum_{i=1}^s \sum_{j=3}^4 \binom{\theta}{j}M^{\theta-j}A_i^j+\sum_{i=1}^s r_4(A_i)\Big|. \]
Rearranging and using the bounds $r_4(A_i) \ll M^{\theta-5/2}$ and $A_i \ll M^{1/2}$ we may see that 
\[ \binom{\theta}{2}M^{\theta-2}\sum_{i=1}^s A_i^2 \ll \tau + \sum_{j=3}^4 M^{\theta-j/2-1}\sum_{i=1}^s A_i^2+M^{\theta-5/2} \ll 1 +M^{\theta-5/2}\sum_{i=1}^s A_i^2+M^{\theta-5/2} . \]
Since \eqref{eq:A_i} holds and $\theta > 2$, this is impossible, since the terms on the right are too small to majorize the term on the left.  

Now instead suppose $s$ is odd. Since $c < (\lfloor s/2 \rfloor(\theta-1))^{-1/2}$ we may choose $\delta > 0$ such that $c < (\lfloor s/2 \rfloor(\theta-1)(1+\delta)^3)^{-1/2}$. Since $s$ is odd, $\lfloor s/2 \rfloor = (s-1)/2$ and by \eqref{eq:sA_i} we must have 
\[ \Big|s-\sum_{i=1}^s A_i\Big| \leq \frac{1}{2}sc^2(1+\delta)^2(\theta-1)+o(1) < \frac{s}{(1+\delta)(s-1)}+o(1) < 2+o(1). \]
Define $\chi$ to be the expression within the absolute values on the left hand side. For sufficiently large $M$ it must then be that $\chi \in \{ -1, 0 , 1\}$. If $\chi = 0$, then we may argue exactly as in the previous case to derive the same contradiction. If instead $\chi = \pm 1$, then inserting this into \eqref{eq:comp} with $k = 2$ and using the Triangle inequality shows 
\[ |\theta M^{\theta-1}\chi| < \tau +\Big| \binom{\theta}{2} M^{\theta-2} \sum_{i=1}^s A_i^2\Big|+O(M^{\theta-3/2}), \]
whence 
\begin{align} \label{eq:1bound}
1 < \frac{\theta-1}{2M}\sum_{i=1}^sA_i^2+O(M^{-1/2})
\end{align}
Since $s$ is odd, one of the two sets $\{1 \leq i \leq s: A_i > 0\}$ and $\{1 \leq i \leq s :A_i \leq 0\}$ contains at most $(s-1)/2$ elements. Let $U$ be the smaller set, and let $V$ be the larger. It follows that 
\[ \sum_{i \in U} |A_i| \leq 2^{-1}c(s-1)(1+\delta)M^{1/2} \]
since $|A_i| \leq c(1+\delta)M^{1/2}$ for all $i$. Note that all $A_i$ with $i \in V$ have the same sign. Hence, by definition of $\chi$ one has 
\[ \sum_{i \in V} |A_i| = \big|\sum_{i \in V} A_i\big| = \big|s-\chi-\sum_{i \in U} A_i\big| \leq s+1+2^{-1}c(s-1)(1+\delta)M^{1/2}. \] 
Moreover, we have 
\[ \sum_{i=1}^s |A_i| = \sum_{i \in U} |A_i| +\sum_{i \in V} |A_i| \leq c(s-1)(1+\delta)M^{1/2}+s+1 \leq c(s-1)(1+\delta)^2M^{1/2} \]
for all sufficiently large $M$. Therefore, 
\[ \sum_{i=1}^s |A_i|^2 \leq c(1+\delta)M^{1/2}\sum_{i=1}^s |A_i| \leq c^2(s-1)(1+\delta)^3M. \]
Using this in \eqref{eq:1bound} we deduce 
\[ 1 < \frac{1}{2}c^2(\theta-1)(s-1)(1+\delta)^3+O(M^{-1/2}). \]
However, the upper bound on $c$ implies that 
\[ \frac{1}{2}c^2(\theta-1)(s-1)(1+\delta)^3 < 1, \]
and so the above inequality is impossible if $M$ is too large, which is a contradiction. 
\end{proof} 

Theorem \ref{thm:neg} turns out the be surprising because it does not match a heuristic for the size of $Y$. Let $x_i = X+y_i$, where $-Y < y_i \leq Y$ for $i = 1,\ldots,s$. Then the number of choices for $(y_1,\ldots,y_s)$ is $\asymp Y^s$. Since $X-Y < x_i \leq X+Y$, the Mean Value Theorem shows that $x_i^{\theta} - X^{\theta} \ll X^{\theta-1}Y $. Thus, one might expect solutions to exist as soon as $X^{\theta-1}Y \ll Y^s$, which is to say when
\[ Y \gg X^{\frac{\theta-1}{s-1}}. \]
However, Theorem \ref{thm:neg} shows that this heuristic fails to materialize. 

\section{Proof of Theorem \ref{thm:pos}: Preliminaries and The Davenport-Heilbronn Method } \label{sec:prelim}

We now come to Theorem \ref{thm:pos}, which comprises the remainder of this paper. To this end, we start with a combinatorial computation resembling the argument used in Section \ref{sec:neg}. A similar proof strategy was used in \cite{DaemenAA1}; see Section 2 in \cite{DaemenAA1}. Write $y_i = x_i-X$, so that $|y_i| \ll \sqrt{X}$ for all $i$. Then by Taylor expansion (with $k$ to be chosen in terms of $\theta$ later), one has 
\[ \Big|R-\sum_{i=1}^s x_i^{\theta}\Big| = \Big|R-X^{\theta}\sum_{i=1}^s (1+y_i/X)^{\theta}\Big| = \Big|R-X^{\theta}\sum_{i=1}^s \sum_{j=0}^k \binom{\theta}{j} \Big(\frac{y_i}{X}\Big)^j+X^{\theta}\sum_{i=1}^s r_k(y_i/X)\Big| \]
\[ = \Big|R-sX^{\theta}-\theta X^{\theta-1}\sum_{i=1}^s y_i -  \sum_{i=1}^s \sum_{j=2}^k \binom{\theta}{j}X^{\theta-j}y_i^j-X^{\theta}\sum_{i=1}^s r_k(y_i/X)\Big| \]
\[ \geq \Big|\theta X^{\theta-1}\sum_{i=1}^s y_i\Big| - |sX^{\theta}-R|-\Big|\sum_{i=1}^s\sum_{j=2}^k \binom{\theta}{j}X^{\theta-j}y_i^j\Big|-\Big|\sum_{i=1}^s X^{\theta}r_k(y_i/X)\Big|, \]
where $r_k$ again represents the remainder term. We know that \[r_k(y_i/X) = \binom{\theta}{k+1}(1+\xi)^{\theta-k-1}(y_i/X)^{k+1}\] 
for some $\xi$ between $0$ and $y_i/X$. Since $y_i/X \ll X^{-1/2}$, it holds that $1 + \xi = 1+O(X^{-1/2})$, and so we may surmise that as long as $k \geq 1$, one has 
\[\Big|X^{\theta}\sum_{i=1}^s r_k(y_i/X)\Big| \ll X^{\theta}\sum_{i=1}^k |y_i/X|^{k+1} \ll X^{\theta-k/2-1/2} \ll X^{\theta-1}.\] 
Therefore, if one has $|x_1^{\theta}+\cdots+x_s^{\theta}-R| < \tau$, we may use the identities $R=sX^{\theta}$ and $y_i \ll X^{1/2}$ to conclude that 
\[ \Big|\theta X^{\theta-1}\sum_{i=1}^s y_i\Big| < \tau +\Big|\sum_{i=1}^s \sum_{j=2}^k \binom{\theta}{j} X^{\theta-j}y_i^j\Big|+\Big|X^{\theta}\sum_{i=1}^s r_k(y_i/X)\Big| \ll_{\theta,s,\tau} \sum_{i=1}^s \sum_{j=2}^k X^{\theta-j/2} + X^{\theta-1}, \]
from which we conclude that 
\[ \Big|\theta X^{\theta-1}\sum_{i=1}^s y_i\Big| \ll X^{\theta-1}, \]
and so $|y_1+\cdots+y_s| \ll 1$. Hence, $|x_1+\cdots+x_s-sX| \ll 1$. As such, we will find it useful to count the number of solutions $(x_1,\cdots,x_s)$ with $X-Y < x_i \leq X+Y$ to the system 
\begin{equation} \label{eq:system}
\begin{aligned}
|x_1^{\theta}+\cdots+x_s^{\theta} - R| < \tau, \\ 
x_1+\cdots+x_s = sX+ m,
\end{aligned}
\end{equation}
for some suitably chosen real number $m$. The computation above shows that for every solution $(x_1,\cdots, x_s)$ to the first inequality in \eqref{eq:system}, the second equation must hold for some $m \ll 1$, and so it stands to reason that the number of solutions to \eqref{eq:system} will not be much smaller than the number of solutions to just the first inequality, at least if $m$ is chosen correctly. In fact, we may expect that the sizes of the two solution sets are of the same order of magnitude in terms of $X$ and $Y$. We actually have more control over the value of $m$ in this case than in the original context of integer degrees, allowing us to show existence of solutions with a somewhat smaller value of $c$ than in the case of integer degree. Let $N_{s,\theta}^{\tau}(R,Y,m) = N(R,Y,m)$ denote the number of solutions to \eqref{eq:system}. 

We will proceed via (Freeman's version of) the Davenport-Heilbronn method, which is the variant of the Hardy-Littlewood method used for counting solutions to Diophantine inequalities. For more information on Freeman's adaptation of the Davenport-Heilbronn method, see Freeman \cite{FreemanLB, FreemanAsymp} or the exposition by Wooley \cite{WooleyDHF}. For the application of this method to approximation problems with real degree, see Poulias \cite{Poulias1, Poulthesis}. See also Arkhipov-Zhitkov \cite{ArkhZhit} and Deshouillers \cite{Des} for related results. One may refer to Chapter 20 in \cite{DavBook} for a demonstration of the classical version of this method.  
We start with a lemma about a useful auxiliary function, the kernel function $K$. For any measurable set $S$ we write $\chi_S$ to denote the indicator function of $S$. 
\begin{lemma} \label{lem:kernel} 
Fix $h \in \mathbb{N}$ and $a, b \in \mathbb{R}$ with $0 < a < b$. Then there is an even real function $K(\alpha) = K(\alpha;a,b)$ such that 
\[ K(\alpha) \ll_h \min \{b, |\alpha|^{-1}, |\alpha|^{-h-1}(b-a)^{-h} \} \]
and the function $\psi$ defined by 
\[ \psi(\xi) = \int_{\mathbb{R}} e(\xi\alpha)K(\alpha) \rm{d}\alpha \]
is such that $\chi_{[-a,a]}(\xi) \leq \psi(\xi) \leq \chi_{(-b,b)}(\xi) $.
\end{lemma}
\begin{proof} This is Lemma 1 in \cite{FreemanAsymp}. 
\end{proof} 
Let $\tilde{\tau} = \tau \log(X)^{-1}$. In this setup, we can actually choose the kernel function with even more information available to us. We put 
\begin{align} \label{eq:kerform}
K_{\pm}(\alpha) = \frac{\sin(\pi\alpha\tilde{\tau})\sin(\pi\alpha(2\tau\pm\tilde{\tau}))}{\pi^2 \alpha^2 \tilde{\tau}} = (2\tau\pm\tilde{\tau})\frac{\sin(\pi \alpha \tilde{\tau})}{\pi\alpha\tilde{\tau}} \cdot \frac{\sin(\pi\alpha(2\tau\pm \tilde{\tau}))}{\pi \alpha (2\tau\pm\tilde{\tau})} .
\end{align} 
One may check that $K_{\pm}(\alpha)$ satisfies the conclusion one Lemma \ref{lem:kernel} with $h = 1$; see Lemma 1 and its proof in \cite{FreemanAsymp} or Section 3.2.1 in \cite{Poulthesis}. 
We therefore record the estimate 
\begin{align} \label{eq:kerbound}
K_{\pm}(\alpha) \ll_{\tau} \min\{1, |\alpha|^{-1},\log(X)|\alpha|^{-2} \}
\end{align}
which will be used later on. Moreover, applying the approximation $\sin(x)/x = 1+O(x^2)$ for $|x| < 1$ with $x \neq 0$ to \eqref{eq:kerform} allows us to deduce that 
\begin{align} \label{eq:kerasymp}
K_{\pm}(\alpha) = 2\tau+O(\log(X)^{-1}) 
\end{align}
for $|\alpha| < X^{-\delta}$ for any fixed $\delta > 0$. Equation \eqref{eq:kerasymp} will be particularly useful for the major arc analysis. 
See also Section 3.2.1 in \cite{Poulthesis} for more on this approximation. 
Letting $\psi_{\pm}(\xi)$ be the functions associated with $K_{\pm}(\alpha)$ as referred to in Lemma \ref{lem:kernel}, we have 
\[ \chi_{(-\tau+\tilde{\tau},\tau-\tilde{\tau})} (\xi) \leq \psi_-(\xi) \leq \chi_{(-\tau,\tau)}(\xi) \leq \psi_+(\xi) \leq \chi_{(-\tau-\tilde{\tau},\tau+\tilde{\tau})}(\xi). \]
Moreover, we have that $|\psi_{\pm}(\xi)-\chi_{(-\tau,\tau)}(\xi)| = 0$ whenever $||\xi|-\tau|>\tilde{\tau}$ and is at most $1$ when $||\xi|-\tau|\leq \tilde{\tau}$. 
For $\alpha_1 \in [-1/2,1/2]$ and $\alpha_2 \in \mathbb{R}$, let 
\[ f(\alpha_1,\alpha_2) = \sum_{X-Y < x \leq X+Y} e\big((x-X)\alpha_1+\alpha_2(x^{\theta}-X^{\theta}-\theta (x-X)X^{\theta-1})\big). \]
For a measurable set $\Omega \subseteq [-1/2,1/2]\times \mathbb{R}$, define 
\[ H_{\pm}(R,m;\Omega) = \int_{\Omega} f(\alpha_1,\alpha_2)^se\big(-m\alpha_1-\alpha_2(R-sX^{\theta}-\theta mX^{\theta-1})\big)K_{\pm}(\alpha_2)\d\alpha_2 \d\alpha_1 \]
and let $H_{\pm}(R,m) = H_{\pm}(R,m;[-1/2,1/2]\times \mathbb{R})$. Expanding the definition of $f(\alpha_1, \alpha_2)$ and using orthogonality and Fubini's Theorem, we have 
\[ H_{\pm}(R,m) = \sum_{\substack{X-Y < x_1,\cdots,x_s \leq X+Y \\ x_1+\cdots+x_s=sX+m}} \; \int_{-\infty}^{\infty} e\Big(\alpha_2\big(\sum_{i=1}^s x_i^{\theta}-R\big)\Big)K_{\pm}(\alpha_2)\d\alpha_2. \]
Therefore, by the definition of $\psi_{\pm}$ and the inequalities $\psi_-(\xi) \leq \chi_{(-\tau,\tau)}(\xi) \leq \psi_+(\xi)$, one uses the simple counting observation 
\[ N_{s,\theta}^{\tau}(R,Y,m) = \sum_{\substack{X-Y < x_1,\cdots,x_s \leq X+Y \\ x_1+\cdots+x_s=sX+m}} \, \chi_{(-\tau,\tau)}\Big(\sum_{i=1}^s x_i^{\theta}-R\Big) \]
to arrive at the fundamental identity 
\[ H_-(R,m) \leq N_{s,\theta}^{\tau}(R,Y,m) \leq H_+(R,m). \]
Therefore, we essentially need to prove sufficient lower bounds on $H_-(R,m)$ to prove the needed conclusion. At times, it can be useful to have the similar facts available for $H_+(R,m)$ as well, which is why we also define this generating function, even though at face value the only one we really need is $H_-(R,m)$. In practice, $H_-$ and $H_+$ are essentially interchangeable and all the same estimates hold for both as far as we are concerned, and so we do not strongly distinguish between them and simply work with $H_{\pm}$ and $K_{\pm}$ in general, unless specified otherwise.

The next important task is the decomposition of the infinite strip of integration into the major, minor and trivial arcs. Define the major arcs by 
\[ \mathfrak{M} = \big\{ (\alpha_1,\alpha_2) \in [-1/2,1/2)\times \mathbb{R} : |\alpha_2| \leq 2(5\theta(\theta-1))^{-1}Y^{-1}(X+Y)^{-\theta+2} \big\}, \] 
the minor arcs 
\[ \mathfrak{m} = \big\{ (\alpha_1,\alpha_2) \in [-1/2,1/2)\times \mathbb{R}: 2(5\theta(\theta-1))^{-1} Y^{-1}(X+Y)^{-\theta+2} < |\alpha_2| \leq X^{\nu} \big\} \] 
where $\nu = 2^{-\theta-5}$, and finally the trivial arcs \[ \mathfrak{t} = \{ (\alpha_1, \alpha_2) \in [-1/2,1/2) \times \mathbb{R}: |\alpha_2| > X^{\nu} \}.\] 
As these are disjoint sets, we have $H_{\pm}(R,m) = H_{\pm}(R,m; \mathfrak{M}) + H_{\pm}(R,m; \mathfrak{m}) +H_{\pm}(R,m; \mathfrak{t})$. We will prove non-trivial upper bounds for the integrals over the minor and trivial arcs and a lower bound on the major arcs. 

\section{Minor Arc Bounds} \label{sec:min}

\subsection{Mean Value Estimates} 
The objective of this section is to derive mean value estimates which will be used to derive estimates on the minor arcs. We begin with a lemma, which is similar to Lemma 1 from \cite{ArkhZhit}. For ease of notation, for any integer $i$ define 
\[ b_i = \binom{\theta}{i}. \]

\begin{lemma} \label{lem:arkzhit}
Let $\theta > 2$ and let $m$ be an integer with $m \geq 2\theta $, and let $P > 0$ be a large enough real number. Define $\mathcal{H} = b_2P^{\theta-2}x_2+\cdots+b_mP^{\theta-m}x_m$. Then for any positive real number $t$, the number $T(P)$ of solutions $(x_2, \cdots, x_m)$ with $|x_i| \ll_{i, \theta} tP^{i/2}$ to the inequality 
\[ |\mathcal{H}| \leq 2t \]
satisfies 
\[ T(P) \ll_{\theta} t^mP^{m(m+1)/4-\theta+1/2}. \]
\end{lemma}

\begin{proof} The proof is very similar to the proof of Lemma 1 in \cite{ArkhZhit}. If $|\mathcal{H}| \leq 2t$, then we must have 
\[ b_2P^{\theta-2}x_2+\cdots+b_mP^{\theta-m}x_m = 2\eta t \]
for some $|\eta| \leq 1$. 
As a result, one has 
\[ x_2 = b_2^{-1}(-b_3x_3P^{-1}-b_4x_4P^{-2}-\cdots-b_mx_mP^{2-m}+2t\eta P^{2-\theta}) = b_2^{-1} \Big( 2t\eta P^{2-\theta} - \sum_{i=3}^m b_ix_iP^{-i+2} \Big). \]
Since $|x_i| \ll tP^{i/2}$ for $i = 2, \ldots, m$, we thus have 
\[ |x_2| \ll |t||b_2|^{-1} \Big( 2P^{2-\theta} + \sum_{i=3}^m |b_i|P^{2-i/2} \Big) \ll |t|(P^{1/2}+P^{2-\theta}). \]
Since $\theta > 2$, the right hand side above is $\ll |t|P^{1/2}$, and so for large $P$, we can deduce that $x_2$ can take on at most $\ll tP^{1/2}$ many values. 

Now, choose an integer $i > 2$. Let 
\[ A_i = -b_i^{-1}\sum_{j = 2}^{i-1} b_jP^{\theta-j}x_j. \]
Then by the fact that $\mathcal{H} = 2\eta t$ we must have 
\[ x_i - A_i = -b_i^{-1}\sum_{j=i+1}^m b_jP^{i-j}x_j +2t\eta b_i^{-1}P^{i-\theta}, \]
so that 
\[ |x_i - A_i| \ll b_i^{-1}\Big( 2tP^{i-\theta} + \sum_{j=i+1}^m |b_j||x_j|P^{i-j} \Big) \ll |t|(P^{i-\theta}+P^{(i-1)/2}). \]
Fix the choices of $x_2, \ldots, x_{i-1}$, so that $A_i$ is fixed. If $i- \theta \leq (i-1)/2$, i.e. if $i \leq 2\theta-1$, then $|x_i -A_i| \ll tP^{(i-1)/2}$ and so there are $\ll tP^{(i-1)/2}$ many values that $x_i$ can take once all previous variables are fixed. Instead, if $i > 2\theta-1$, then $|x_i-A_i| \ll tP^{i-\theta} $, and therefore the number of values that $x_i$ can take is $\ll tP^{i-\theta}$. However, we also have the trivial estimate from assumption that $x_i$ can take at most $\ll tP^{i/2}$ many values, which is a better estimate than $\ll tP^{i-\theta}$ as long as $i/2 \leq i-\theta$, i.e. when $2\theta \leq i$. As long as $m \geq 2\theta$, there will be a unique integer $r$ such that $2\theta-1 < r \leq 2\theta$. The total number of solutions to the above inequality is hence $\ll t^m P^{\kappa}$, where 
\[ \kappa = \sum_{2 \leq i \leq r-1} (i-1)/2+\sum_{r+1 \leq i \leq m} i/2 + r-\theta = m(m+1)/4+1/2 - \theta. \]
Putting together the above gives us the required estimate. 
\end{proof}

Now we need one more lemma, which follows from the same method of proof as Lemma 3.2 in \cite{Poulias1}. See also Lemma 2.1 in \cite{Watt}. For a real number $\delta$, let $V_s(\delta)$ be the number of solutions to the system 
\begin{align*} \Big|\sum_{i=1}^s (x_i^{\theta}-x_{i+s}^{\theta})\Big| < \delta, 
\\ 
x_1+\cdots+x_s = x_{i+s}+\cdots+x_{2s} 
\end{align*} 
with $x_i \in (X-Y,X+Y]$ for all $i$. Then the following holds. 
\begin{lemma} \label{lem:count}
Let $\delta$ and $\Delta$ be positive real numbers such that $2\Delta\delta = 1$ and let $s$ be a positive integer. Then 
\[ \delta \int_{-1/2}^{1/2} \int_{-\Delta}^{\Delta} |f(\alpha_1,\alpha_2)|^{2s} \d\alpha_2\d\alpha_1 \leq \frac{\pi^2}{4}V_s( \delta). \]
\end{lemma}

\begin{proof} The proof method is essentially the same as in \cite{Poulias1}, but we supply the details for completeness. Define the auxiliary functions $K_0(\alpha) = \operatorname{sinc}^2(\alpha) $ and $\Lambda(x) = \max\{ 0, 1-|x|\}$, where $\operatorname{sinc}(\alpha) = 1 $ if $\alpha =0$ and $\operatorname{sinc}(\alpha) = (\pi\alpha)^{-1}\sin(\pi\alpha)$
otherwise. We then have the well-known identities
\[ K_0(\xi) = \int_{\mathbb{R}} e(-x\xi)\Lambda(x)\d x \quad \text{ and } \quad \Lambda(x) = \int_{\mathbb{R}} e(x\xi)K_0(\xi)\d\xi, \]
as one may see from Lemma 20.1 in \cite{DavBook}. One also has the easily verifiable inequality that $\pi^2K_0(\alpha)/4 \geq 1$ for $|\alpha|\leq 1/2$. For $\mathbf{x} = (x_1,\ldots,x_{2s})$, put 
\[ \sigma_{s,\theta}(\mathbf{x}) = \sum_{i=1}^s \big(x_i^{\theta}-x_{i+s}^{\theta}\big) \quad \text{ and } \quad \xi_0 = (2\delta)^{-1} \sigma_{s,\theta}(\mathbf{x}). \]
Note that for $-\Delta \leq \alpha \leq \Delta$, one has $|\delta\alpha|\leq 1/2$, and so from the estimate on $K_0$ and non-negativity of the integrand we have 
\[ \int_{-\Delta}^{\Delta} |f(\alpha_1,\alpha_2)|^{2s}\d\alpha_2 \leq \frac{\pi^2}{4}\int_{-\Delta}^{\Delta} |f(\alpha_1,\alpha_2)|^{2s}K_0(\delta\alpha_2)\d\alpha_2 \leq \frac{\pi^2}{4}\int_{-\infty}^{\infty}|f(\alpha_1,\alpha_2)|^{2s}K_0(\delta\alpha_2)\d\alpha_2. \]
Inserting this estimate into the left hand side $L$ in the statement of the lemma, we have by orthogonality and Fubini's Theorem that 
\[ L \leq \sum_{\substack{X-Y < x_1,\ldots,x_{2s} \leq X+Y \\ x_1+\cdots +x_s = x_{s+1}+\cdots +x_{2s}}} \frac{\pi^2\delta}{4} \int_{-\infty}^{\infty} e(\alpha_2 \sigma_{s,\theta}(\mathbf{x}))K_0(\delta\alpha_2)\d\alpha_2 = \frac{\pi^2}{4}\sum_{\mathbf{x}} \int_{-\infty}^{\infty} e(2u\xi_0)K_0(u)\d u, \]
where the summation in the right-hand side is over the same range. Note that we used the substitution $u = \delta\alpha$ and the definition of $\xi_0$. Finally, using the relation between $K_0$ and $\Lambda$ and the inequality $0 \leq \Lambda(2\xi_0) \leq \chi_{(-1,1)}(2\xi_0)$, we may estimate the sum on the right above as 
\[ \sum_{\mathbf{x}} \int_{-\infty}^{\infty} e(2u\xi_0)K_0(u)\d u = \sum_{\mathbf{x}} \Lambda(2\xi_0) \leq \sum_{\mathbf{x}} \chi_{(-1,1)}(2\xi_0) = V_s(\delta). \]
\end{proof}

We now arrive at the important mean value estimate that we are after. 

\begin{lemma} \label{lem:mean}
For any integer $s \geq (\lfloor 2\theta \rfloor + 1)(\lfloor 2\theta \rfloor + 2)/2$, any real number $\kappa \geq 1$, and any $\epsilon > 0$, we have 
\[ \int_{-1/2}^{1/2} \int_{-\kappa}^{\kappa} |f(\alpha_1,\alpha_2)|^{2s} \d\alpha_2 \d\alpha_1 \ll \kappa X^{s-\theta+1/2}Y^{\epsilon}.  \]
\end{lemma}

\begin{proof} By Lemma \ref{lem:count}, 
we have 
\[ \frac{1}{2\kappa}\int_{-1/2}^{1/2} \int_{-\kappa}^{\kappa} |f(\alpha_1,\alpha_2)|^{2s} \d\alpha_2 \d\alpha_1 \ll V_s\Big( \frac{1}{2\kappa}\Big), \]
and since $\kappa \geq 1$ we trivially have that $V_s(1/(2\kappa)) \leq V_s(1/2)$, and so it suffices to show that $V_s(1/2) \ll X^{s-\theta+1/2}Y^{\epsilon}$. Let $T = \lfloor X-Y\rfloor$ and $z_i = x_i-T$. Thus, $1 \leq z_i \leq 2Y+1$. 
Using a Taylor expansion with $k$ to be chosen later in terms of $\theta$, one has 
\begin{align*} \Big|\sum_{i=1}^s ((T+z_i)^{\theta}-(T+z_{i+s})^{\theta})\Big| & = \Big|\sum_{i=1}^s \sum_{j=0}^k \binom{\theta}{j}T^{\theta}\big((z_i/T)^j-(z_{i+s}/T)^j\big) + T^{\theta}\sum_{i=1}^k r_k(z_i/T) \Big| \\ 
& = \Big|\sum_{j=1}^k \binom{\theta}{j}T^{\theta-j}\sum_{i=1}^s (z_i^j-z_{i+s}^j) +T^{\theta}\sum_{i=1}^s r_k(z_i/T)\Big|. 
\end{align*}
Now, as before, we note that $r_k(z_i/T) \ll |z_i/T|^{k+1} \ll X^{-k/2-1/2}$. Therefore, if we choose $k$ such that $k > 2\theta-1$, then 
\[ \Big|T^{\theta}\sum_{i=1}^s r_k(z_i/T)\Big| \ll X^{\theta-k/2-1/2} = o(1). \]
For large $X$, we then see that the remainder is $< 1/2$. As a result, $V_s(1/2)$ is bounded above by the number of solutions to the system 
\begin{align*} 
\Big|\sum_{j=1}^k \binom{\theta}{j}T^{\theta-j}\sum_{i=1}^s (z_i^j-z_{i+s}^j)\Big| < 1, \\ 
z_1+\cdots+z_s = z_{s+1}+\cdots+z_{2s}
\end{align*}
for large $X$. Using the second equation in the first inequality, the inequality reduces to 
\[ \Big|\sum_{j=2}^k \binom{\theta}{j}T^{\theta-j}\sum_{i=1}^s (z_i^j-z_{i+s}^j)\Big| < 1. \]
Therefore, the number of solutions to the system above is bounded above by the number of solutions to the system 
\begin{align*} 
 & \Big|\sum_{j=2}^k \binom{\theta}{j}T^{\theta-j}h_j\Big| < 1, \\ 
& \sum_{i=1}^s (z_i^j-z_{i+s}^j) = h_j \quad \quad (1 \leq j \leq k),
\end{align*}
where $1 \leq z_i \leq 2Y+1$ and $h_1 = 0$. Call this number $Z_{s,k}(Y)$, and let $\mathbf{h} = (h_1, \cdots, h_k)$. Note that clearly we must have $|h_j| \leq s(2Y+1)^j$ for each $2 \leq j \leq k$. Let $J_{s,k}(Y; \mathbf{h})$ be the number of solutions to the system 
\[ \sum_{i=1}^s (z_i^j - z_{i+s}^j) = h_j \quad \quad (1 \leq j \leq k). \]
In the special case $\mathbf{h} = 0$ write this quantity as $J_{s,k}(Y)$. Then by orthogonality and use of the triangle inequality, one has $J_{s,k}(Y, \mathbf{h}) \leq J_{s,k}(Y)$, and applying the now resolved Main Conjecture of Vinogradov's Mean Value Theorem (see Theorem 1.1 in \cite{BDG} and the results of \cite{WooNEC}), whenever $s \geq k(k+1)/2$ one has 
\[ J_{s,k}(Y, \mathbf{h}) \leq J_{s,k}(Y) \ll Y^{2s-k(k+1)/2+\epsilon}. \]
Therefore, letting 
\[ \mathcal{F}(\mathbf{h}) = \sum_{j=2}^k \binom{\theta}{j} T^{\theta-j}h_j, \] 
we see 
\[ Z_{s,k}(Y) \leq \sum_{\substack{|\mathcal{F}(\mathbf{h})| < 1 \\ |h_j| \leq s(2Y+1)^{j} \\ 2 \leq j \leq k}} J_{s,k}(Y; \mathbf{h}) \ll Y^{2s-k(k+1)/2+\epsilon} \sum_{\substack{|\mathcal{F}(\mathbf{h})| < 1 \\ |h_j| \leq s(2Y+1)^{j} \\ 2 \leq j \leq k}} 1. \]
Our final step is to apply Lemma \ref{lem:arkzhit}, which supplies a bound for the number of solutions to the inequality $|\mathcal{F}(\mathbf{h})| < 1<s$, where $h_1 = 0$ and $|h_j| \leq s(2Y+1)^j$ for each $2 \leq j \leq k$. Applying this with $k = m = \lfloor 2\theta \rfloor + 1$ and noting that $Y \asymp X^{1/2} \asymp T^{1/2}$, we see 
\[ Z_{s,k}(Y) \ll Y^{2s-k(k+1)/2+\epsilon}X^{m(m+1)/4-\theta+1/2} \ll X^{s-\theta+1/2}Y^{\epsilon}. \]
Since we only require $s \geq k(k+1)/2 = (\lfloor 2\theta \rfloor + 1)(\lfloor 2\theta \rfloor + 2)/2$, this completes the proof. 
\end{proof}

\subsection{Weyl-type bounds on the Minor Arcs} 

The next task is to prove $L^{\infty}$ bounds for $f(\alpha_1,\alpha_2)$ on $\mathfrak{m}$. We will make use of a classical version of Van der Corput's $k$th derivative test, which we formulate below. 
Before we proceed, we must note that the bounds we prove here are quite crude. One can almost certainly improve the estimates we prove using different or more modern estimates. However, it is probably also the case that such proofs will be significantly longer and more involved. Because even weak non-trivial estimates suffice for our purposes, we do not pursue stronger estimates here, even though these stronger bounds may be of independent interest. We do however mention some relevant literature for interested readers; for classical versions of Van der Corput's method one may refer to the book by Graham and Kolesnik \cite{GraKol}, and for a more modern view of Van der Corput's method see Robert \cite{Robert}. A classical treatment of Vinogradov's method applied to exponential sums related to $f(\alpha_1,\alpha_2)$ is given in the book by Vinogradov \cite{Vino} while modern developments in the theory of Vinogradov's method are used by K\"{u}fner \cite{Kuf} to prove bounds on a similar exponential sum. 
\begin{lemma} \label{lem:VDC1}
Let $\beta \geq 1$ be a real number and $k \geq 2$ be an integer. There exists $C(\beta, k) > 0$ such that for any integer $M \geq 1$, any real number $\lambda_k$ and any $C^k$ function $g: [1, M] \to \mathbb{R}$ satisfying
\[ \lambda_k \leq |g^{(k)}(x)| \leq \beta \lambda_k \quad \text{ for } \quad x \in [1, M], \]
one has 
\[ \Big|\sum_{1 \leq x \leq M} e(g(x))\Big| \leq C(\beta, k) (M \lambda_k^{1/(2^k-2)} + M^{1-2^{2-k}} \lambda_k^{-1/(2^k-2)}). \]
\end{lemma} 

\begin{proof} See Theorem 3 in \cite{Robert}. 
\end{proof} 
For another, similar formulation of Van der Corput's test, see Theorem 2.8 in \cite{GraKol}. Recalling that $T = \lfloor X-Y \rfloor$, we apply this lemma to the function $g(t) = \alpha_1 t +\alpha_2\big((t+T)^{\theta}-X^{\theta}-\theta tX^{\theta-1}\big)$. 

\begin{lemma} \label{lem:weyl}
Let $(\alpha_1, \alpha_2) \in \mathfrak{m}$, so that $ 2 (5\theta(\theta-1))^{-1}Y^{-1}(X+Y)^{-\theta+2} \leq |\alpha_2| \leq X^{\nu} $. 
Then 
\[ |f(\alpha_1,\alpha_2)| = \Big|\sum_{1 \leq x \leq X+Y-T } e(g(x)) \Big| \ll Y^{1-\frac{1}{2^{\theta+\nu+3}}}. \]
\end{lemma} 

\begin{proof} The first equality is obvious. We now calculate the derivatives of $g$. We see 
\[ g'(x) = \alpha_1 + \alpha_2 \big( \theta(x+T)^{\theta-1}-\theta X^{\theta-1} \big). \]
For $k \geq 2$, one has 
\[ g^{(k)}(x) = \alpha_2 \theta(\theta-1)\cdots (\theta-k+1)(x+T)^{\theta-k}. \]
At this point, we may assume without loss of generality that $\alpha_2 \geq 0$, since flipping the sign of $\alpha_2$ has no effect on the magnitude of $g^{(k)}(x)$. Then since $Y \asymp X^{1/2}$, we see that for large $X$,
\[ |g^{(k)}(x)| = \alpha_2 |\theta \cdots (\theta-k+1)|(T+x)^{\theta-k} \geq \alpha_2 |\theta \cdots (\theta-k+1)| (X/2)^{\theta-k}.  \]
At the same time, one also has 
\[ |g^{(k)}(x)| = \alpha_2 |\theta \cdots (\theta-k+1)|(T+x)^{\theta-k} \leq \alpha_2 |\theta \cdots (\theta-k+1)| (2X)^{\theta-k}. \]
Therefore, we may let $\lambda_k = \alpha_2 |\theta \cdots (\theta-k+1)| (X/2)^{\theta-k}$ and one has 
\[ \lambda_k \leq |g^{(k)}(x)| \leq 4^{\theta-k}\lambda_k. \]
We may replace the upper bound $X+Y-T$ in the sum with $2Y$ at the cost of at most $O(1)$ terms. We then have by Lemma \ref{lem:VDC1} with $\lambda_k$ as above and $\beta = 4^{\theta-k}$ that 
\[ \Big|\sum_{1 \leq x \leq 2Y} e(g(x)) \Big| \leq C(\beta, k) \big( (2Y)\lambda_k^{\frac{1}{2^k-2}} + (2Y)^{1-2^{2-k}} \lambda_k^{-1/(2^k-2)} \big). \]
Note that $C(\beta,k)$ only depends on $\theta$ and $k$. So, we may call this constant $C(\theta, k)$.

Let $\lambda = \log(\alpha_2)/\log(X) $. Since $Y \asymp X^{1/2}$, it follows that $ X^{\nu} \geq |\alpha_2| \gg X^{-\theta+3/2}$, and so $\nu \geq \lambda \geq -\theta+1.4$ for large $X$. We consider two cases, the first being when $-\theta+1.4 \leq \lambda \leq -\theta+1.75$ and the second being when $-\theta + 1.75 \leq \lambda \leq \nu$. In the first case, we apply Van der Corput's $k$th derivative test with $k = 2$. Of course we then have $2^k-2 = 2$ and $2^{2-k}=1$. Notice also that $-0.6 \leq \theta+\lambda-k \leq -1/4 $. Since $\alpha_2X^{\theta-k} = X^{\lambda+\theta-k}$ and $X \asymp Y^2$, we then have the bounds $\alpha_2 X^{\theta-k} \ll X^{-1/4} \ll Y^{-1/2}$, and $\alpha_2X^{\theta-k} \gg X^{-3/5} \gg Y^{-6/5}$. Therefore, we have 
\[ \Big|\sum_{1 \leq x \leq 2Y} e(g(x))\Big| \ll_{\theta} \big( Y(\alpha_2X^{\theta-k})^{\frac{1}{2}}+(\alpha_2X^{\theta-k})^{-\frac{1}{2}}\big) \ll Y^{1-1/4}+(Y^{6/5})^{1/2} \ll Y^{1-\frac{1}{2^\theta}}, \]
where in the last line we used that $\theta > 2$. Obviously, $Y^{1-\frac{1}{2^{\theta}}} \ll Y^{1-\frac{1}{2^{\theta+\nu+3}} } $ and so the claim holds in this case. 

In the case $-\theta+1.75 \leq \lambda \leq \nu$, let $k = \lfloor \theta+\lambda +1.25 \rfloor $. Note that $3 \leq k \leq \theta+\nu+1.25 \leq 2\theta$. Also, $-1.25 \leq \lambda+\theta-k \leq -0.25$, since $-1.25 \leq t-
\lfloor t+1.25 \rfloor \leq -0.25 $ for all $t \in \mathbb{R}$. Let $C^*(\theta) = \max_{3 \leq i \leq 3\theta} \{C(\theta, i)\}$. As such, we then have 
\begin{align*}
\Big|\sum_{1 \leq x \leq 2Y} e(g(x)) \Big| & \leq  C^*(\theta) \big( (2Y)(\alpha_2X^{\theta-k})^{\frac{1}{2^k-2}} + (2Y)^{1-2^{2-k}} (\alpha_2 X^{\theta-k})^{-1/(2^k-2)} \big) 
\\ & \leq  2C^*(\theta)\big( Y(\alpha_2X^{\theta-k})^{\frac{1}{2^k-2}} + Y^{1-\frac{4}{2^k}}(\alpha_2X^{\theta-k})^{-\frac{1}{2^k-2}} \big).  
\end{align*}
Again since $ \alpha_2X^{\theta-k} = X^{\lambda+\theta-k}$, and $-1.25 \leq\lambda+\theta-k \leq -0.25$, we have the bounds 
\[ Y^{-2.5} \ll X^{-1.25} \ll  \alpha_2X^{\theta-k} \leq X^{-0.25} \ll  Y^{-0.5}. \]
Therefore, we see 
\[ \Big|\sum_{1 \leq x \leq 2Y} e(g(x)) \Big| \leq 2C^*(\theta) \big( c^{\frac{1}{2^k-2}} Y^{1-\frac{1/2}{2^k-2}} + c^{-\frac{1}{2^k-2}} Y^{1-\frac{4}{2^k}+\frac{5/2}{2^k-2}} \big). \]
Now we may use the bounds $3 \leq k \leq 2\theta$ to conclude that there is some constant $B = B(\theta)$ such that 
\[ \Big|\sum_{1 \leq x \leq 2Y} e(g(x))\Big| \leq B(\theta) \big(Y^{1-\frac{1/2}{2^k-2}} + Y^{1-\frac{4}{2^k}+\frac{5/2}{2^k-2}} \big). \]
Now, note that 
\[ 1-\frac{4}{2^k}+\frac{5/2}{2^k-2} \leq 1-\frac{1}{2^{1+k}} \quad \text{ and } \quad 1-\frac{1/2}{2^k-2} \leq 1-\frac{1}{2^{1+k}}\] 
for $k \geq 3$, and since $k \leq \theta+\nu+2$ we may conclude that 
\[ \Big|\sum_{1 \leq x \leq 2Y} e(g(x))\Big| \ll Y^{1-\frac{1}{2^{1+k}}} \ll_{\theta} Y^{1-\frac{1}{2^{\theta+\nu+3}}}. \]
\end{proof}

As we mentioned earlier, these bounds are weak compared to the bounds which may be available with current technology. With additional effort, one suspects that one may be able to use the techniques of \cite{Kuf} or \cite{Vino} to prove a bound of the shape $f(\alpha_1,\alpha_2) \ll Y^{1-\sigma(\theta)^{-1}}$ where $\sigma(\theta)$ is a quadratic polynomial in $\theta$.

\subsection{Bounds on the minor arcs}
We are now ready to estimate $H_{\pm}(R,m)$ on the minor arcs. Let $s \geq (\lfloor 2\theta \rfloor+1)(\lfloor 2\theta \rfloor+2)+1$, and let $2r$ be the largest even integer strictly smaller than $s$. By inspection, one then sees that $2r \geq (\lfloor 2\theta \rfloor+1)(\lfloor 2\theta \rfloor+2)$, and so for any $\kappa \geq 1$ we may apply Lemma \ref{lem:mean} to deduce that
\[ \int_{-1/2}^{1/2} \int_{-\kappa}^{\kappa} |f(\alpha_1,\alpha_2)|^{2r}\d\alpha_2\d\alpha_1 \ll \kappa X^{r-\theta+1/2}Y^{\epsilon} \ll \kappa X^{1-\theta} Y^{2r-1+\epsilon}. \]
Therefore, noting that $(\alpha_1,\alpha_2) \in \mathfrak{m}$ implies $|\alpha_2| \leq X^{\nu}$, one has 
\[ \int_{\mathfrak{m}} |f(\alpha_1,\alpha_2)|^{s}\d\alpha_2\d\alpha_1 \leq \sup_{(\alpha_1,\alpha_2) \in \mathfrak{m} } |f(\alpha_1,\alpha_2)|^{s-2r} \int_{\mathfrak{m}} |f(\alpha_1,\alpha_2)|^{2r}\d\alpha_2\d\alpha_1 \] 
\[\ll Y^{(s-2r)(1-2^{-(\theta+\nu+3)})} X^{\nu}X^{1-\theta}Y^{2r-1+\epsilon} \ll X^{1-\theta}Y^{s-1+\epsilon}Y^{2\nu-2^{-(\theta+\nu+3)}(s-2r)}. \]
Finally, since $s-2r \geq 1$, we know $2\nu-2^{-(\theta+\nu+3)}(s-2r) \leq 2\nu - 2^{-(\theta+\nu+3)}$, and using the definition of $\nu$ we have 
\[ 2\nu - 2^{-(\theta+\nu+3)} = 2^{-\theta-4}-2^{-\theta- 2^{-\theta-5} -3} = 2^{-\theta-3}(2^{-1}-2^{-2^{-\theta-5}}) \leq -2^{-\theta-5} = -\nu.  \]
Using the estimate $K_{\pm}(\alpha) \ll 1$, we thus have the following bound, which we collect as a lemma. 
\begin{lemma} \label{lem:minor} 
For any $s \geq (\lfloor 2\theta \rfloor+1)(\lfloor 2\theta \rfloor+2)+1$ and $\epsilon > 0$, we have 
\[ |H_{\pm}(Y,m; \mathfrak{m})| 
\ll \int_{\mathfrak{m}} |f(\alpha_1,\alpha_2)|^{s}\d\alpha_2\d\alpha_1 \ll X^{1-\theta}Y^{s-1-\nu+\epsilon}. \]
\end{lemma}

\section{Trivial Arc Bounds} \label{sec:triv}
We now bound integral over the trivial arcs, the definition of which we recall as 
\[ \mathfrak{t} = \{ (\alpha_1, \alpha_2) \in [-1/2,1/2] \times \mathbb{R}: |\alpha_2| \geq X^{\nu} \} .\] 
The main ingredient is the bound on the kernel function $K_{\pm}$ afforded to us by \eqref{eq:kerbound}. By the change of variables $u_1 = -\alpha_1$ and $u_2 = -\alpha_2$, noting that $K_{\pm}$ is even and the integrand is nonnegative, we use \eqref{eq:kerbound} to deduce 
\[ \int_{\mathfrak{t}} |f(\alpha_1,\alpha_2)|^s|K_{\pm}(\alpha_2)| \d\alpha_2\d\alpha_1 =  2\int_{-1/2}^{1/2}\int_{ X^{\nu}}^{\infty} |f(\alpha_1,\alpha_2)|^s|K_{\pm}(\alpha_2)|\d\alpha_2\d\alpha_1 \]
\[ \ll \sum_{j = \lfloor{\nu \log_2 X \rfloor} }^{\infty} \frac{\log(X)}{2^{2j}}\int_{-1/2}^{1/2}\int_{2^j}^{2^{j+1}} |f(\alpha_1,\alpha_2)|^s \d\alpha_2\d\alpha_1 . \] 
Define $I_j$ as the integral inside the sum. Now suppose $s \geq s_0 = (\lfloor 2\theta \rfloor+1)(\lfloor 2\theta \rfloor+2)$. Note that $s_0$ is even and $2^{j+1} \geq 1$ for $j \geq \lfloor \nu \log_2 X \rfloor$. Therefore, we may use the trivial bound $|f(\alpha_1,\alpha_2)| \ll Y$ and Lemma \ref{lem:mean} to see that 
\[ I_j \ll Y^{s-s_0}\int_{-1/2}^{1/2} \int_{2^j}^{2^{j+1}} |f(\alpha_1,\alpha_2)|^{s_0}\d\alpha_2\d\alpha_1 \ll Y^{s-s_0+\epsilon}2^{j+1}X^{s_0/2-\theta+1/2} \ll 2^{j+1}X^{1-\theta} Y^{s-1+\epsilon}.  \]
Inserting this bound into the above, one has 
\[ \sum_{j = \lfloor{\nu \log_2 X \rfloor} }^{\infty} \frac{\log(X)}{2^{2j}} I_j \ll X^{1-\theta}Y^{s-1+\epsilon}\sum_{j=\lfloor{\nu \log_2 X \rfloor}}^{\infty} 2^{-j} \ll X^{1-\theta}Y^{s-1+\epsilon}X^{-\nu}. \]
Therefore, we have the following. 
\begin{lemma} \label{lem:trivial} 
For any $s \geq (\lfloor 2\theta \rfloor+1)(\lfloor 2\theta \rfloor+2)$ and $\epsilon > 0$, one has 
\[ H_{\pm}(Y,m; \mathfrak{t}) \ll \int_{\mathfrak{t}} |f(\alpha_1,\alpha_2)|^{s}|K_{\pm}(\alpha_2)|\d\alpha_2\d\alpha_1 \ll X^{1-\theta}Y^{s-1-2\nu+\epsilon}. \]
\end{lemma}

\section{Major Arcs} \label{sec:maj}
Finally, we prove non-trivial lower bounds on the major arcs. We note that this section only requires $s \geq \max\{ 2\theta, 8\}$. We being with a general lemma on integral approximations to sums. 

\begin{lemma} \label{lem:poisson}
Let $F: [X, Y] \to \mathbb{R} $ be a continuous, twice differentiable function with $F'$ monotonic on $[X,Y]$. Suppose that $H_1$ and $H_2$ are integers such that $H_1 \leq F'(x) \leq H_2$ for $x \in [X,Y]$. Then one has 
\[ \sum_{A < x \leq B} e(F(x)) = \sum_{h = H_1}^{H_2} \int_A^B e(F(\beta)-\beta h)\d\beta + O(\log(2+H)), \]
where $H = \max\{ |H_1|, |H_2|\}$. 
\end{lemma} 

\begin{proof} See Lemma 4.2 in the book by Vaughan \cite{Vbook}
\end{proof}

Let $g(x) = \alpha_1 x +\alpha_2\big((X+x)^{\theta}-X^{\theta}-\theta xX^{\theta-1} \big) $. Using Lemma \ref{lem:poisson}, we have the following approximation result on the major arcs.

\begin{lemma} \label{lem:approx}
Let $|\alpha_2| \leq 2(5\theta(\theta-1))^{-1} Y^{-1}(X+Y)^{-\theta+2}$. One has 
\[ f(\alpha_1,\alpha_2)= \sum_{X-Y <  x \leq X+Y} e(g(x-X)) = \int_{-Y}^{Y} e(g(t))\d t + O(1). \]
\end{lemma}

\begin{proof} The first equality follows by definition. We see $g'(x-X) = \alpha_1 +\theta\alpha_2\big(x^{\theta-1}- X^{\theta-1}\big) $, which is clearly a monotonic function of $x$ (one may easily check this by taking the second derivative). Furthermore, by the Mean Value Theorem one has 
\begin{align*} |g'(x-X)| \leq |\alpha_1| + |\alpha_2|\theta\big(x^{\theta-1}-X^{\theta-1}\big) & \leq 1/2+\theta(\theta-1)|\alpha_2||x-X|(X+Y)^{\theta-2} \\ &
\leq 1/2 + \theta(\theta-1)|\alpha_2|Y(X+Y)^{\theta-2}. 
\end{align*}
By our assumption on $|\alpha_2|$ we have for all $X-Y \leq x \leq X+Y$ that 
\[ |g'(x)| \leq \frac{1}{2}+\frac{2}{5} = \frac{9}{10}. \]
Therefore, we may take $H_1 = -1$ and $H_2 = 1$. 
Applying Lemma \ref{lem:poisson} and changing variables one has 
\[ \sum_{-Y < x \leq Y} e(g(x-X)) - \int_{-Y}^Y e(g(t))\d t = \sum_{\substack{h = -1 \\ h \neq 0}}^1 e(-hX) \int_{-Y}^{Y} e(g(t)-th)\d t + O(1). \]
We also see that for $|h| \geq 1$, one has for all $-Y \leq t \leq Y$ that 
\[ |g'(t)-h| \geq |h| - |g'(t)| \geq |h|-9/10 \geq |h|/10 \geq 1/10. \]
Define the functions  $v(t) = 2\pi i(g'(t)-h)$ and $u(t) =v(t)^{-1}e(g(t)-th)$. Since $g''(t)$ never changes sign, we integrate by parts for $h = \pm 1$ using $u(t)$ and $v(t)$ to see that 
\begin{align*} \int_{-Y}^Y e(g(t)-th)\d t = \left[ \frac{e(g(t)-th)}{2\pi i (g'(t)-h)} \right]_{-Y}^Y +\int_{-Y}^Y \frac{g''(t)}{(g'(t)-h)^2}\frac{e(g(t)-th)}{2\pi i}\d t \\
\ll \max \{ |g'(-Y)-h|^{-1}, |g'(Y)-h|^{-1} \} \ll 1. 
\end{align*}
Substituting this estimate into the above derives the needed estimate. 
\end{proof}

Our next objective is to substitute $f(\alpha_1,\alpha_2)$ with a suitable integral approximation (as suggested in the previous lemma) on the major arcs. First, we need an estimate on $f^*(\alpha_1, \alpha_2)$, the major arc approximation defined as 
\[ f^*(\alpha_1, \alpha_2) = \int_{-Y}^Y e\big(\alpha_1t+\alpha_2( (X+t)^{\theta}-X^{\theta}-\theta t X^{\theta-1} )\big)\d t = \int_{-Y}^Y e(g(t))\d t. \]
We will use Van der Corput's Stationary Phase Lemma, which we state now. 
\begin{lemma} \label{lem:VDCstat} 
Let $\lambda > 0$. Suppose that $\phi: (a,b) \to \mathbb{R}$ is a smooth function in $(a,b)$ and suppose $|\phi^{(k)}(x)| \geq 1$ for all $x \in (a,b)$. Then 
\[ \Big|\int_a^b e^{i\lambda \phi(t)} \d t \Big| \leq c_k\lambda^{-1/k} \]
as long as either $k \geq 2$, or if $k = 1$ and $\phi'(t)$ is monotonic. Moreover, the constant $c_k$ is independent of $\phi$ and $\lambda$. 
\end{lemma}

\begin{proof} See Proposition 2 on page 332 of \cite{Stein}. See also \cite{VDC}. 
\end{proof}

\begin{lemma} \label{lem:apprest}
For any $\beta_1, \beta_2 \in \mathbb{R}$, we have 
\[ f^*(\beta_1, \beta_2) \ll Y(1+|\beta_1|Y+|\beta_2|X^{\theta-2}Y^2)^{-1/2}, \]
where the implicit constants are independent of $\beta_i$, $X$ and $Y$.
\end{lemma}

\begin{proof} 
We see 
\[ f^*(\beta_1, \beta_2) = Y \int_{-1}^1 e\big(\beta_1Yt+\beta_2((X+Yt)^{\theta}-X^{\theta}-\theta YtX^{\theta-1} )\big)\d t. \]
Let $I_1$ be the integral on the right hand side. We show that $I_1 \ll (1+|\beta_1|Y+|\beta_2|X^{\theta-2}Y^2)^{-1/2}$, which will complete the proof. We assume 
\[|\beta_1|Y+|\beta_2|Y^{2}X^{\theta-2} \geq 1,\] 
since otherwise the needed estimate is trivial. We consider two cases. Suppose first that $|\beta_2|X^{\theta-2}Y^2 \geq |\beta_1|Y$. In this case, letting $\phi(t)$ be the phase in $I_1$, we may notice that 
\[ \phi''(t) = \theta(\theta-1)\beta_2Y^2(X+Yt)^{\theta-2}, \]
and therefore for $-1 \leq t \leq 1$ we have 
\[ |\phi''(t)| \geq \theta(\theta-1)|\beta_2|Y^2(X-Y)^{\theta-2} \geq C_{\theta}|\beta_2|Y^2(X+Y)^{\theta-2} \]
for some constant $C_\theta$. Then using Lemma \ref{lem:VDCstat} with $k = 2$ and $\lambda = C_{\theta}|\beta_2|Y^2(X+Y)^{\theta-2}$ and the function $\phi_0(t) = \lambda^{-1}\phi(t)$, we see that 
\[ \int_{-1}^1 e(\phi(t))\d t = \int_{-1}^1 e(\lambda\phi_0(t)) \d t \ll |\lambda|^{-1/2} \ll (|\beta_2|Y^2(X+Y)^{\theta-2})^{-1/2}.  \]
Then, since $1 \leq |\beta_1|Y+|\beta_2|Y^{2}X^{\theta-2} \leq 2|\beta_2|Y^{2}X^{\theta-2}$, 
we have 
\[ 1+|\beta_1|Y+|\beta_2|Y^2(X+Y)^{\theta-2} \leq 4|\beta_2|Y^2(X+Y)^{\theta-2}, \] 
from which we have the needed estimate $I_1 \ll (1+|\beta_1|Y+|\beta_2|Y^2X^{\theta-2})^{-1/2}$ by rearranging and recalling that $X+Y \asymp X$. 
In the second case, assume $|\beta_1|Y \geq |\beta_2|X^{\theta-2}Y^2$. We have by the Mean Value Theorem that 
\[ |\phi'(t)| = |\beta_1Y+\theta\beta_2Y((X+Yt)^{\theta-1}-X^{\theta-1})| \geq |\beta_1|Y-\theta(\theta-1)|\beta_2|Y(X+Y)^{\theta-2}. \]
First, suppose also that $|\beta_1|Y \geq 2\theta(\theta-1)|\beta_2|Y(X+Y)^{\theta-2}$. Then we have 
\[ |\phi'(t)| \geq |\beta_1|Y/2, \]
and so letting $\lambda = |\beta_1|Y/2$ and $\phi_0(t) = \lambda^{-1}\phi(t)$, we use Lemma \ref{lem:VDCstat} with $k = 1$ to note that
\[ \int_{-1}^1 e(\phi(t)\d t \ll (|\beta_1|Y)^{-1}. \]
Then since $1 \leq |\beta_1|Y+|\beta_2|Y^2(X+Y)^{\theta-2} \leq 2|\beta_1|Y$, we have 
\[ (1+|\beta_1|Y+|\beta_2|Y^2(X+Y)^{\theta-2})^{1/2} \leq (4|\beta_1|Y)^{1/2} \leq 4|\beta_1|Y, \]
and again the needed estimate on $I_1$ follows. 

If $|\beta_1|Y \leq 2\theta(\theta-1)|\beta_2|X^{\theta-2}Y^2$ instead, we apply Lemma \ref{lem:VDCstat} with $k = 2$ to deduce that 
\[ \int_{-1}^1 e(\phi(t))\d t \ll (|\beta_2|Y^2(X+Y)^{\theta-2})^{-1/2}. \]
In this case, $1 \leq |\beta_1|Y + |\beta_2|Y^2(X+Y)^{\theta-2} \leq  C_{\theta}|\beta_2|Y^2(X+Y)^{\theta-2}$ for some $C_{\theta}$.
Therefore, 
\[ 1+|\beta_1|Y+|\beta_2|Y^2(X+Y)^{\theta-2} \leq  2C_{\theta}|\beta_2|Y^2(X+Y)^{\theta-2}. \]
Rearranging this inequality gives required estimate. 
\end{proof}

From Lemma \ref{lem:approx} we have $f(\alpha_1,\alpha_2) -f^*(\alpha_1,\alpha_2) \ll 1$ on $\mathfrak{M}$. Let $Z = R-sX^{\theta}-\theta m X^{\theta-1}$ and assume $s \geq 8$. Let 
\[ J_{\pm}^*(R,m) = \int_{\mathfrak{M}} f^*(\alpha_1, \alpha_2)^se(-m\alpha_1-\alpha_2Z)K_{\pm}(\alpha_2)\d\alpha_2 \d\alpha_1. \]
\begin{lemma} \label{lem:intapprox}
Let $\theta > 2$ and $s \geq \max\{2\theta,8\}$. For sufficiently large $R$ we have 
\[ H_{\pm}(R,m;\mathfrak{M}) = \int_{\mathfrak{M}} f(\alpha_1, \alpha_2)^se(-m\alpha_1-Z\alpha_2)K_{\pm}(\alpha_2)\d\alpha_2\d\alpha_1 = J_{\pm}^*(R,m) +O(X^{1-\theta}Y^{s-2}). \]
\end{lemma}

\begin{proof} The first equality is by definition. By the Binomial theorem, $H_{\pm}(Y,m;\mathfrak{M})$ expands to 
\[ \sum_{i=0}^s \binom{s}{i} \int_{\mathfrak{M}} f^*(\alpha_1,\alpha_2)^{s-i}(f(\alpha_1,\alpha_2)-f^*(\alpha_1,\alpha_2))^{s-i}e(-m\alpha_1-Z\alpha_2)K_{\pm}(\alpha_2)\d\alpha_1\d\alpha_2. \]
Since $f(\alpha_1, \alpha_2)-f^*(\alpha_1,\alpha_2) \ll 1$ on $\mathfrak{M}$ and $K_{\pm} \ll 1$, we see that the above reduces to 
\[ \int_{\mathfrak{M}} f^*(\alpha_1, \alpha_2)^se(-m\alpha_1-Z\alpha_2)K_{\pm}(\alpha_2)\d\alpha_1\d\alpha_2 +O\Big(\sum_{i=1}^s \int_{\mathfrak{M}} |f^*(\alpha_1,\alpha_2)|^{s-i}\d\alpha_1\d\alpha_2\Big). \]
The first term above is exactly $J_{\pm}^*(R,m)$, so we only need to bound the errors. If $s-i \leq 4$, then by the Triangle inequality we have 
\[ \int_{\mathfrak{M}} |f^*(\alpha_1,\alpha_2)|^{s-i}\d\alpha_1\d\alpha_2 \ll Y^{s-i}|\mathfrak{M}| \ll Y^4Y^{-1}X^{-\theta+2} \ll X^{1-\theta}Y^{5} \ll X^{1-\theta}Y^{s-2} \]
since $s \geq 7$.
For $i = 1, \cdots, s-5$, we see by Lemma \ref{lem:apprest} that 
\begin{equation} \label{eq:intbound}
\begin{aligned}
\int_{\mathfrak{M}}|f^*(\alpha_1, \alpha_2)|^{s-i} \d\alpha_1\d\alpha_2 & \ll 
Y^{s-i} \int_{-1/2}^{1/2} \int_{-\infty}^{\infty} (1+|\beta_1|Y+|\beta_2|X^{\theta-2}Y^2)^{-(s-i)/2}\d\beta_1\d\beta_2  
\\ & \leq Y^{s-i} \int_{-\infty}^{\infty} \int_{-\infty}^{\infty} (1+|\beta_1|Y+|\beta_2|X^{\theta-2}Y^2)^{-(s-i)/2}\d\beta_1\d\beta_2. 
\end{aligned}
\end{equation}
Using the elementary inequality 
\begin{align} \label{eq:elemineq}
(1+u+v)^{-1} \leq (1+u)^{-1/2}(1+v)^{-1/2} \quad \text{ for } \quad u,v \geq 0,
\end{align} 
we have that the final display in \eqref{eq:intbound} is 
\begin{align*} & \leq 4Y^{s-i}\int_0^{\infty} \frac{\d\beta_1}{(1+\beta_1Y)^{\frac{s-i}{4}}} \int_0^{\infty} \frac{\d\beta_2}{(1+\beta_2X^{\theta-2}Y^2)^{\frac{s-i}{4}}} 
\\ & \ll Y^{s-i} Y^{-1}X^{2-\theta}Y^{-2}\int_0^{\infty} \frac{du_1}{(1+u_1)^{\frac{s-i}{4}}} \int_0^{\infty} \frac{du_2}{(1+u_2)^{\frac{s-i}{4}}} \ll Y^{s-i-1}X^{1-\theta}, 
\end{align*}
since by assumption $(s-i)/4 \geq 5/4$. Combining these estimates, we also see in this case that 
\[ \int_{\mathfrak{M}}|f^*(\alpha_1, \alpha_2)|^{s-i} \d\alpha_1\d\alpha_2 \ll Y^{s-i-1}X^{1-\theta} \ll X^{1-\theta}Y^{s-2}. \]
\end{proof}

By Lemma \ref{lem:intapprox}, it suffices to show $ J_{\pm}^*(R,m) \gg \tau X^{1-\theta}Y^{s-1}$. To do this, we first approximate $J_{\pm}^*(R,m)$ with another integral. Let 
\[ J(R,m) = \int_{\mathbb{R}^2} f^*(\alpha_1,\alpha_2)^se(-m\alpha_1-Z\alpha_2)\d\alpha_1\d\alpha_2. \]

\begin{lemma} \label{lem:sing} Let $\theta > 2$ and $s \geq \max\{2\theta,8\}$. For sufficiently large $R$ we have 
\[ J_{\pm}^*(R,m) = 2\tau J(R,m) + O(Y^{s-1}X^{1-\theta}\log(X)^{-1}).  \]
\end{lemma}

\begin{proof} First, let 
\[ J^*(R,m) = \int_{\mathfrak{M}} f^*(\alpha_1, \alpha_2)^se(-m\alpha_1-Z\alpha_2)\d\alpha_2\d\alpha_1. \]
Using \eqref{eq:kerasymp}, we note that 
\[ J^*_{\pm}(R,m) = \int_{\mathfrak{M}} f^*(\alpha_1,\alpha_2)^se(-m\alpha_1-\alpha_2Z)(2\tau+O(\log(X)^{-1}))\d\alpha_2\d\alpha_1 \] 
\[ = 2\tau J^*(R,m) + O\Big((\log(X)^{-1}\int_{\mathbb{R}^2} |f^*(\alpha_1,\alpha_2)|^s\d\alpha_2\d\alpha_1\Big). \]
Using Lemma \ref{lem:apprest} and the inequality \eqref{eq:elemineq} gives 
\[ \int_{\mathbb{R}^2} |f^*(\alpha_1,\alpha_2)|^s\d\alpha_2\d\alpha_1 \ll Y^s\int_{\mathbb{R}^2} (1+|\alpha_1|Y+|\alpha_2|X^{\theta-2}Y^2)^{-s/2}\d\alpha_2\d\alpha_1 \]
\[ \ll Y^s \int_{0}^{\infty} \frac{\d\alpha_1}{(1+\alpha_1Y)^{s/4}} \int_{-0}^{\infty} \frac{\d\alpha_2}{(1+\alpha_2X^{\theta-2}Y^2)^{s/4}}. \]
Putting $\beta_1 = \alpha_1Y$ and $\beta_2 = \alpha_2X^{\theta-2}Y^2$, we see that the above is $ \ll Y^{s-1}X^{2-\theta}Y^{-2} \ll Y^{s-1}X^{1-\theta}$.
This estimate shows us that $J_{\pm}^*(R,m) = 2\tau J^*(R,m)+O(X^{1-\theta}Y^{s-1}\log(X)^{-1})$. 

The estimates above also show that $J(R,m)$ is an absolutely convergent integral, and that 
\[ J(R,m) \ll Y^s \int_{0}^{\infty} \frac{\d\beta_1}{(1+\beta_1Y)^{s/4}} \int_{0}^{\infty} \frac{\d\beta_2}{(1+\beta_2 X^{\theta-2}Y^2)^{s/4}} \ll Y^{s-1}X^{1-\theta}. \]
We then also have the bound 
\[ J(R,m)-J^*(R,m) = \int_{\mathbb{R}^2 \setminus [-1/2, 1/2) \times [-U, U]} f^*(\alpha_1,\alpha_2)^s e(-\alpha_2Z-\alpha_1m)\d\alpha_2\d\alpha_1, \]
where $U = 2(5\theta(\theta-1))^{-1}Y^{-1}(X+Y)^{-\theta+2}$. The right hand side above can be bounded as 
\begin{align*} \ll & \; Y^s\int_{-1/2}^{1/2} \int_{|\alpha_2| > U} (1+|\alpha_1|Y+|\alpha_2|(X+Y)^{\theta-2}Y^2)^{-s/2}\d\alpha_2\d\alpha_1  \\
& \; + Y^s\int_{|\alpha_1|\geq 1/2} \int_{-\infty}^{\infty} (1+|\alpha_1|Y+|\alpha_2|(X+Y)^{\theta-2}Y^2)^{-s/2}\d\alpha_2\d\alpha_1.  
\end{align*}
Call the first term $A_1$ and the second $A_2$. We estimate these terms separately. Note that the integrand is non-negative and even in both terms. 

To estimate $A_1$, the inequality \eqref{eq:elemineq} shows that 
\begin{align*} A_1 & \ll Y^s\int_{0}^{\infty} (1+\alpha_1Y)^{-s/4} \d\alpha_1 \int_{U}^{\infty} (1+\alpha_2X^{\theta-2}Y^2)^{-s/4}\d\alpha_2 \\ 
& = Y^sY^{-1}X^{2-\theta}Y^{-2}\int_{0}^{\infty}(1+\beta_1)^{-s/4}\d\beta_1 \int_{UX^{{\theta-2}}Y^{2}}^{\infty} (1+\beta_2)^{-s/4}\d\beta_2. 
\end{align*}
The integral in $\beta_1$ is $\ll 1$ since $s \geq 5$. By the definition of $U$, the integral in $\beta_2$ is 
\[ \ll (UY^2(X+Y)^{\theta-2})^{1-s/4} \ll Y^{1-s/4}. \]
Therefore, since $s \geq 8$ we see that 
\[ A_1 \ll Y^{s-2-s/4}X^{2-\theta} \ll Y^{s-s/4}X^{1-\theta} \ll Y^{s-2}X^{1-\theta}, \]
where we used $Y^2 \asymp X$.  
By the same reasoning, 
\[ A_2 \ll Y^{s}Y^{-1}X^{2-\theta}Y^{-2} \int_{Y/2}^{\infty} (1+\beta_1)^{-s/4}\d\beta_1 \int_{-\infty}^{\infty}(1+\beta_2)^{-s/4}\d\beta_2. \]
In this case the integral in $\beta_2$ is $\ll 1$ and the integral in $\beta_1$ is $\ll Y^{1-s/4}$. Thus, 
\[ A_2 \ll Y^{s-2-s/4}X^{2-\theta} \ll Y^{s-2}X^{1-\theta}. \]

Putting everything together, we see $J^*(R,m) = J(R,m)+O(X^{1-\theta}Y^{s-2}) $, and so 
\[ J^*_{\pm}(R,m) = 2\tau J^*(R,m)+O(X^{1-\theta}Y^{s-1}\log(X)^{-1}) = 2\tau J(R,m)+O(X^{1-\theta}Y^{s-1}\log(X)^{-1}). \]
\end{proof}

Collecting all estimates together from Lemmas \ref{lem:minor}, \ref{lem:trivial}, \ref{lem:intapprox} and \ref{lem:sing} , we see that 
\[ H_{\pm}(R,m) = 2\tau J(R,m) + O(X^{1-\theta}Y^{s-1}\log(X)^{-1}). \]
Therefore, to complete the proof it suffices to show that  $J(R,m) \gg X^{1-\theta}Y^{s-1}$.
By change of variables, one has 
\[ J(R,m) = Y^s \int_{\mathbb{R}^2} \Big( \int_{-1}^1 e\big(\alpha_1Yt +\alpha_2((X+Yt)^{\theta}-X^{\theta}-\theta YtX^{\theta-1})\big)dt\Big)^se(-m\alpha_1-Z\alpha_2)\d\alpha_2\d\alpha_1.  \]
Let $G(t) = (X+t)^{\theta}-X^{\theta}-\theta tX^{\theta-1}$. Putting $\beta_1 = Y\alpha_1$ and $\beta_2 = X^{\theta-2}Y^2\alpha_2$, we see 
\[ J(R,m) = X^{2-\theta}Y^{s-3}J_0(R,m), \]
where 
\[ J_0(R,m) = \int_{\mathbb{R}^2} \Big( \int_{-1}^1 e(\beta_1t+\beta_2X^{2-\theta}Y^{-2}G(Yt))dt \Big)^s e(-m\beta_1/Y-Z\beta_2X^{2-\theta}Y^{-2})\d\beta_2\d\beta_1. \]
Since $X^{2-\theta}Y^{s-3} \gg X^{1-\theta}Y^{s-1}$, it only remains to show that $J_0(R,m) \gg 1$ when $R$ is sufficiently large. Expanding, we see 
\[ J_0(R,m) = \int_{\mathbb{R}^2} \int_{[-1,1]^s} e(\beta_1u(\boldsymbol{\gamma})+\beta_2v(\boldsymbol{\gamma}))\d\mathbf{\gamma} \d\beta_1\d\beta_2, \]
where 
\[ u(\boldsymbol{\gamma}) = \sum_{i=1}^s \gamma_i-\frac{m}{Y} \quad \text{ and } \quad v(\boldsymbol{\gamma}) = \sum_{i=1}^s \frac{G(Y\gamma_i)}{X^{\theta-2}Y^2}-\frac{Z}{X^{\theta-2}Y^2}. \]

We will now analyze the solution set to the system $u(\boldsymbol{\gamma}) = v(\boldsymbol{\gamma}) = 0$. We will first show that if $m \neq 0$ then any real solution $\boldsymbol{\eta} = (\eta_1,\ldots,\eta_s)$ must be non-singular. 
We start by showing that if $m \neq 0$, then one cannot have $\eta_1 = \cdots = \eta_s$. 
Suppose by way of seeking a contradiction that $\boldsymbol{\eta}$ is a solution with $\eta_1 = \cdots = \eta_s$. Since $u(\boldsymbol{\eta}) = 0$, we must have $\eta_i = m/(sY)$ for each $i$. Then using the definition of $G(t)$ and recalling that $Z = R-sX^{\theta}-\theta m X^{\theta-1}$ gives
\begin{align*} X^{\theta-2}Y^2v(\boldsymbol{\eta}) = \sum_{i=1}^s G(Y\eta_i) - Z = \sum_{i=1}^s (X+Y\eta_i)^{\theta}-R = s(X+m/s)^{\theta}-R. 
\end{align*}
If $m \neq 0$, then since $R = sX^{\theta}$, the right hand side above cannot equal $0$ and so we cannot have $v(\boldsymbol{\eta}) = 0$, meaning that $\boldsymbol{\eta}$ is not a solution, which is a contradiction. 
Therefore, if $\boldsymbol{\eta}$ satisfies $u(\boldsymbol{\eta})=v(\boldsymbol{\eta}) = 0$ then there must be some $i,j $ with $i < j$ and $\eta_i \neq \eta_j$. 

Let $A_{i,j}$ be the $2 \times 2$ minor matrix comprising the $i$th and $j$th columns of the Jacobian matrix 
\[ \frac{\partial(u,v)}{\partial(\eta_1,\ldots,\eta_s)}. \]
Then since $\eta_i \neq \eta_j$, one has 
\[ |\det{A_{i,j}}| = \theta Y^{-1}X^{2-\theta}|(X+Y\eta_i)^{\theta-1}-X^{\theta-1}-((X+\eta_jY)^{\theta-1}-X^{\theta-1}) | \neq 0, \]
showing that the Jacobian matrix has full rank and hence any real solution $\boldsymbol{\eta}$ is non-singular. Thus, if there exists a real solution $\boldsymbol{\eta}$ in the interior of $[-1,1]^s$, its non-singularity allows us to use the Implicit Function Theorem (see, for example, Theorem 13.7 in \cite{Apo}) to conclude that there is an $(s-2)$-dimensional subspace $L$ of positive $(s-2)$-volume in a neighborhood of $\boldsymbol{\eta}$ on which $u(\boldsymbol{\gamma})=v(\boldsymbol{\gamma}) = 0$. Then, since $\boldsymbol{\eta}$ is in the interior of $[-1,1]^s$ it follows from a standard truncation and approximation argument (see, for example, Lemmas 3.6.4, 3.6.5, 3.6.6 and 3.6.7 in section 3.6.1 of \cite{Poulthesis}) that $J_0(R,m) \gg 1$. It therefore only remains to ensure the existence of a real solution $\boldsymbol{\eta}$ to the system $u(\boldsymbol{\gamma}) = v(\boldsymbol{\gamma}) = 0$ in the interior of $[-1,1]^s$. 

Let $\mathcal{B}$ be the set of all $\boldsymbol{\gamma} \in (-1,1)^s $ satisfying $u(\boldsymbol{\gamma}) = 0$. We will show that there exist $\boldsymbol{\eta_-}$ and $\boldsymbol{\eta_+} $ in $\mathcal{B}$ such that $v(\boldsymbol{\eta_-}) \leq 0 \leq v(\boldsymbol{\eta_+})$. By continuity of $v(\boldsymbol{\gamma})$, the Intermediate Value Theorem then allows us to conclude that there exists $\boldsymbol{\eta}$ on the line connecting $\boldsymbol{\eta_-}$ and $\boldsymbol{\eta_+}$ such that $v(\boldsymbol{\eta}) = 0$. Since $(-1,1)^s$ is a convex set, it follows that $\boldsymbol{\eta} \in (-1,1)^s $, meaning that $\boldsymbol{\eta}$ is in the interior of $[-1,1]^s$. Finally, since the set of $\boldsymbol{\gamma}$ satisfying $u(\boldsymbol{\gamma}) = 0$ is a hyperplane containing the points $\boldsymbol{\eta_-}$ and $\boldsymbol{\eta_+}$, the line segment connecting $\boldsymbol{\eta_-}$ and $\boldsymbol{\eta_+}$ also lies on this hyperplane, which in particular means that $u(\boldsymbol{\eta}) = 0$. Thus, $\boldsymbol{\eta}$ satisfies all needed properties, completing the proof.

Recall that $Y = c\sqrt{X}$. By Taylor expansion, we have for any $\boldsymbol{\gamma} \in \mathcal{B}$ that 
\begin{align*} \label{eq:expandv}
v(\boldsymbol{\gamma}) & = \frac{1}{X^{\theta-2}Y^2} \Big(\sum_{i=1}^s \big((X+Y\gamma_i)^{\theta}-X^{\theta}-\theta Y \gamma_i X^{\theta-1}\big)-R+sX^{\theta}+m\theta X^{\theta-1}\Big) 
\\ & = \frac{1}{X^{\theta-2}Y^2} \Big( \sum_{i=1}^s (X+Y\gamma_i)^{\theta}-R\Big) 
\\ & = \frac{1}{X^{\theta-2}Y^2} \Big( sX^{\theta}+\theta X^{\theta-1}Y\sum_{i=1}^s \gamma_i + \binom{\theta}{2}X^{\theta-2}Y^2\sum_{i=1}^s \gamma_i^2 -R+O(X^{\theta-3}Y^3)\Big)
\\ & = \frac{\theta Xm}{Y^2} +\binom{\theta}{2}\sum_{i=1}^s \gamma_i^2 + O\Big(\frac{Y}{X}\Big) = \frac{\theta m }{c^2}+\binom{\theta}{2}\| \boldsymbol{\gamma}\|_{\ell^2}^2+O(X^{-1/2}). 
\end{align*}
Let $\boldsymbol{\eta_-} = (m/(sY),\ldots,m/(sY))$. Clearly, $u(\boldsymbol{\eta_-}) = 0$ and $\boldsymbol{\eta_-} \in (-1,1)^s$ for large $X$, whereby $\boldsymbol{\eta_-} \in \mathcal{B}$. Note that $\| \boldsymbol{\eta_-}\|_{\ell^2}^2 = 4m^2/(sY^2)$. (One may determine using Lagrange Multipliers that $\boldsymbol{\eta_-}$ minimizes $\| \boldsymbol{\gamma} \|_{\ell^2}^2$ for $\boldsymbol{\gamma} \in \mathcal{B}$.) Thus, 
\[ v(\boldsymbol{\eta_-}) = \frac{\theta m }{c^2}+\binom{\theta}{2}\| \boldsymbol{\eta_-}\|_{\ell^2}^2+O(X^{-1/2}) = \frac{\theta m}{c^2}+\binom{\theta }{2}\frac{4m}{sY^2}+O(X^{-1/2}) = \frac{\theta m}{c^2}+O(X^{-1/2}).  \]

We now make use of the value of $c$ for the first time in this proof, as we have not needed to do so until now, to select $m$. Note that so far we have only made the restriction, so long as $sX+m \in \mathbb{Z}$, that $m \neq 0$. Since $c > ((\theta-1)\lfloor s/2\rfloor)^{-1/2}$ by assumption, there is a constant $\delta \in (0,1)$ such that 
\[ c > \frac{1}{\delta\sqrt{ (\theta-1)\lfloor s/2 \rfloor}}. \]
Let $\mu = \min\{1/2,\delta^{-1}-1\}$. Thus, $0 < \mu \leq 1/2$ is a fixed constant depending only on $c$. Let $m = -\{sX\} $ if $-\{sX\} \leq -\mu$ and let $m = -1-\{sX\}$ if $-\{sX\} > -\mu$. Since $\mu \leq 1$, one may check that $-1-\mu \leq m \leq -\mu$. 
Therefore, $v(\boldsymbol{\eta_-}) \leq -\theta \mu/c^2 + O(X^{-1/2}) < 0 $ and is bounded above by a negative constant independent of $X$ for all sufficiently large $X$, as required. 

Next, let $ \boldsymbol{\eta_+} = (\eta_1,\ldots,\eta_s) $ be defined by $\eta_i = \delta\sqrt{1+\mu}(-1)^{i}$ for $i = 1,\ldots,s-1$ and $\eta_s = \delta\sqrt{1+\mu}(1+(-1)^s)/2+mY^{-1}$. Splitting into the cases where $s$ is even or odd, one may check that $u(\boldsymbol{\eta_+}) = 0 $. Next, $\| \boldsymbol{\eta_+}\|_{\ell^{\infty}} \leq \delta\sqrt{1+\mu}+O(Y^{-1}) \leq \sqrt{\delta}+O(Y^{-1}) < 1$ for all sufficiently large $X$, and so $\boldsymbol{\eta_+} \in (-1,1)^s$ for large $X$, whence $\boldsymbol{\eta_+} \in \mathcal{B}$. Now, note that $\| \boldsymbol{\eta_+} \|_{\ell^2}^2 = 2\delta^2(1+\mu)\lfloor s/2 \rfloor+O(Y^{-1})$, and so 
\[ v(\boldsymbol{\eta_+}) = \frac{\theta m }{c^2}+\binom{\theta}{2}\| \boldsymbol{\eta_+}\|_{\ell^2}^2+O(X^{-1/2}) = \frac{\theta m }{c^2}+2\binom{\theta}{2}\delta^2(1+\mu)\lfloor s/2 \rfloor+O(X^{-1/2}). \]
By the lower bound on $c$, and since $m \geq -1-\mu$, we see 
\[ \frac{\theta m}{c^2}+2\binom{\theta}{2}\delta^2(1+\mu)\lfloor s/2 \rfloor > -(1+\mu)\delta^2\theta(\theta-1)\lfloor s/2 \rfloor +2\binom{\theta}{2}\delta^2(1+\mu)\lfloor s/2\rfloor = 0, \]
and so $v(\boldsymbol{\eta_+}) > 0$ for sufficiently large $X$, as needed. Thus, there is a solution $\boldsymbol{\eta} \in (-1,1)^{s}$ to $u(\boldsymbol{\eta})=v(\boldsymbol{\eta}) = 0$. Moreover, $\boldsymbol{\eta} = t\boldsymbol{\eta_-}+(1-t)\boldsymbol{\eta_+}$ for some $t \in [0,1]$, so that 
\[ \|\boldsymbol{\eta} \|_{\ell^{\infty}} \leq t\| \boldsymbol{\eta_-}\|_{\ell^{\infty}}+(1-t)\| \boldsymbol{\eta_+} \|_{\ell^{\infty}} \leq \sqrt{\delta} +O(Y^{-1}), \] 
and so $\|\boldsymbol{\eta} \|_{\ell^{\infty}}$ is bounded away from $1$ for all large $X$. 
Further, 
\[ 0 = v(\boldsymbol{\eta}) = \frac{\theta m}{c^2} +\binom{\theta}{2}\| \boldsymbol{\eta} \|_{\ell^{2}} +O(X^{-1/2}) \leq -\frac{\theta \mu}{c^2} +\binom{\theta}{2}\| \boldsymbol{\eta} \|_{\ell^{2}} +O(X^{-1/2}), \]
so that $\| \boldsymbol{\eta} \|_{\ell^{2}} \geq c_0$ for some constant $c_0 > 0$ independent of $X$ for all sufficiently large $X$. This confirms the existence of a non-singular real solution $\boldsymbol{\eta} \in (-1,1)^s $  for all sufficiently large $X$, proving that $J_0(R,m) \gg 1$. 

\bibliographystyle{plain}
\bibliography{refs} 

@article {DaemenAA1,
    AUTHOR = {Daemen, D.},
     TITLE = {Localized solutions in {W}aring's problem: the lower bound},
   JOURNAL = {Acta Arith.},
  FJOURNAL = {Acta Arithmetica},
    VOLUME = {142},
      YEAR = {2010},
    NUMBER = {2},
     PAGES = {129--143},
      ISSN = {0065-1036,1730-6264},
   MRCLASS = {11P05 (11P55)},
  MRNUMBER = {2601055},
MRREVIEWER = {Karin\ Halupczok},
       DOI = {10.4064/aa142-2-3},
       URL = {https://doi-org.ezproxy.lib.purdue.edu/10.4064/aa142-2-3},
}

@article {WrightAsymp,
    AUTHOR = {Wright, E. M.},
     TITLE = {Proportionality conditions in {W}aring's problem.},
   JOURNAL = {Math. Z.},
  FJOURNAL = {Mathematische Zeitschrift},
    VOLUME = {38},
      YEAR = {1934},
    NUMBER = {1},
     PAGES = {730--746},
      ISSN = {0025-5874,1432-1823},
   MRCLASS = {99-04},
  MRNUMBER = {1545482},
       DOI = {10.1007/BF01170669},
       URL = {https://doi-org.ezproxy.lib.purdue.edu/10.1007/BF01170669},
}

@article {DaemenBLMS,
    AUTHOR = {Daemen, D.},
     TITLE = {The asymptotic formula for localized solutions in {W}aring's problem and approximations to {W}eyl sums},
   JOURNAL = {Bull. Lond. Math. Soc.},
  FJOURNAL = {Bulletin of the London Mathematical Society},
    VOLUME = {42},
      YEAR = {2010},
    NUMBER = {1},
     PAGES = {75--82},
      ISSN = {0024-6093,1469-2120},
   MRCLASS = {11P05 (11L15)},
  MRNUMBER = {2586968},
MRREVIEWER = {Guangshi\ L\"u},
       DOI = {10.1112/blms/bdp095},
       URL = {https://doi-org.ezproxy.lib.purdue.edu/10.1112/blms/bdp095},
}

@article {Poulias1,
    AUTHOR = {Poulias, C.},
     TITLE = {Diophantine inequalities of fractional degree},
   JOURNAL = {Mathematika},
  FJOURNAL = {Mathematika. A Journal of Pure and Applied Mathematics},
    VOLUME = {67},
      YEAR = {2021},
    NUMBER = {4},
     PAGES = {949--980},
      ISSN = {0025-5793,2041-7942},
   MRCLASS = {11D75 (11D72 11L07 11P55)},
  MRNUMBER = {4311791},
MRREVIEWER = {Weiping\ Li},
       DOI = {10.1112/mtk.12112},
       URL = {https://doi-org.ezproxy.lib.purdue.edu/10.1112/mtk.12112},
}

@article{Wright1,
  title={THE REPRESENTATION OF A NUMBER AS A SUM OF FOUR ‘ALMOST EQUAL’ SQUARES},
  author={Wright, E. M. },
  journal={Quarterly Journal of Mathematics},
  volume={8},
  year={1937},
  pages={278-279},
  url={https://api.semanticscholar.org/CorpusID:121158126}
}

@book {Vbook,
    AUTHOR = {Vaughan, R. C.},
     TITLE = {The {H}ardy-{L}ittlewood method},
    SERIES = {Cambridge Tracts in Mathematics},
    VOLUME = {125},
   EDITION = {Second},
 PUBLISHER = {Cambridge University Press, Cambridge},
      YEAR = {1997},
     PAGES = {xiv+232},
      ISBN = {0-521-57347-5},
   MRCLASS = {11P55 (11L15 11P05)},
  MRNUMBER = {1435742},
MRREVIEWER = {D.\ R.\ Heath-Brown},
       DOI = {10.1017/CBO9780511470929},
       URL = {https://doi.org/10.1017/CBO9780511470929},
}

@incollection {FreemanAsymp,
    AUTHOR = {Freeman, D. E. },
     TITLE = {Asymptotic lower bounds and formulas for {D}iophantine
              inequalities},
 BOOKTITLE = {Number theory for the millennium, {II} ({U}rbana, {IL}, 2000)},
     PAGES = {57--74},
 PUBLISHER = {A K Peters, Natick, MA},
      YEAR = {2002},
      ISBN = {1-56881-146-2},
   MRCLASS = {11D75 (11P55)},
  MRNUMBER = {1956244},
MRREVIEWER = {Scott\ T.\ Parsell},
}

@article {FreemanLB,
    AUTHOR = {Freeman, D. E. },
     TITLE = {Asymptotic lower bounds for {D}iophantine inequalities},
   JOURNAL = {Mathematika},
  FJOURNAL = {Mathematika. A Journal of Pure and Applied Mathematics},
    VOLUME = {47},
      YEAR = {2000},
    NUMBER = {1-2},
     PAGES = {127--159},
      ISSN = {0025-5793},
   MRCLASS = {11P55 (11D75)},
  MRNUMBER = {1924493},
MRREVIEWER = {R.\ C.\ Baker},
       DOI = {10.1112/S0025579300015771},
       URL = {https://doi-org.ezproxy.lib.purdue.edu/10.1112/S0025579300015771},
}

@inproceedings {WooleyDHF,
    AUTHOR = {Wooley, T. D.},
     TITLE = {On {D}iophantine inequalities: {F}reeman's asymptotic
              formulae},
 BOOKTITLE = {Proceedings of the {S}ession in {A}nalytic {N}umber {T}heory
              and {D}iophantine {E}quations},
    SERIES = {Bonner Math. Schriften},
    VOLUME = {360},
     PAGES = {1--32},
 PUBLISHER = {Univ. Bonn, Bonn},
      YEAR = {2003},
   MRCLASS = {11D75 (11P55)},
  MRNUMBER = {2075639},
MRREVIEWER = {Scott\ T.\ Parsell},
}

@article {ArkhZhit,
    AUTHOR = {Arkhipov, G. I. and Zhitkov, A. N.},
     TITLE = {Waring's problem with nonintegral exponent},
   JOURNAL = {Izv. Akad. Nauk SSSR Ser. Mat.},
  FJOURNAL = {Izvestiya Akademii Nauk SSSR. Seriya Matematicheskaya},
    VOLUME = {48},
      YEAR = {1984},
    NUMBER = {6},
     PAGES = {1138--1150},
      ISSN = {0373-2436},
   MRCLASS = {11P05},
  MRNUMBER = {772109},
MRREVIEWER = {Ekkehard\ Kr\"atzel},
}

@article {Des,
    AUTHOR = {Deshouillers, J.-M.},
     TITLE = {Probl\`eme de {W}aring avec exposants non entiers},
   JOURNAL = {Bull. Soc. Math. France},
  FJOURNAL = {Bulletin de la Soci\'et\'e{} Math\'ematique de France},
    VOLUME = {101},
      YEAR = {1973},
     PAGES = {285--295},
      ISSN = {0037-9484},
   MRCLASS = {10J10},
  MRNUMBER = {342477},
MRREVIEWER = {B.\ Garrison},
       URL = {http://www.numdam.org/item?id=BSMF_1973__101__285_0},
}

@article {Watt,
    AUTHOR = {Watt, N.},
     TITLE = {Exponential sums and the {R}iemann zeta-function. {II}},
   JOURNAL = {J. London Math. Soc. (2)},
  FJOURNAL = {Journal of the London Mathematical Society. Second Series},
    VOLUME = {39},
      YEAR = {1989},
    NUMBER = {3},
     PAGES = {385--404},
      ISSN = {0024-6107,1469-7750},
   MRCLASS = {11L40 (11M06)},
  MRNUMBER = {1002452},
MRREVIEWER = {Matti\ Jutila},
       DOI = {10.1112/jlms/s2-39.3.385},
       URL = {https://doi-org.ezproxy.lib.purdue.edu/10.1112/jlms/s2-39.3.385},
}

@book {GraKol,
    AUTHOR = {Graham, S. W. and Kolesnik, G.},
     TITLE = {van der {C}orput's method of exponential sums},
    SERIES = {London Mathematical Society Lecture Note Series},
    VOLUME = {126},
 PUBLISHER = {Cambridge University Press, Cambridge},
      YEAR = {1991},
     PAGES = {vi+120},
      ISBN = {0-521-33927-8},
   MRCLASS = {11L07},
  MRNUMBER = {1145488},
MRREVIEWER = {Zun\ Shan},
       DOI = {10.1017/CBO9780511661976},
       URL = {https://doi-org.ezproxy.lib.purdue.edu/10.1017/CBO9780511661976},
}

@article {Robert,
    AUTHOR = {Robert, O.},
     TITLE = {On van der {C}orput's {$k$}-th derivative test for exponential
              sums},
   JOURNAL = {Indag. Math. (N.S.)},
  FJOURNAL = {Koninklijke Nederlandse Akademie van Wetenschappen.
              Indagationes Mathematicae. New Series},
    VOLUME = {27},
      YEAR = {2016},
    NUMBER = {2},
     PAGES = {559--589},
      ISSN = {0019-3577,1872-6100},
   MRCLASS = {11L07 (11L03 40A05 42B35)},
  MRNUMBER = {3479173},
MRREVIEWER = {Tatjana\ L.\ Todorova},
       DOI = {10.1016/j.indag.2015.11.009},
       URL = {https://doi-org.ezproxy.lib.purdue.edu/10.1016/j.indag.2015.11.009},
}

@book {Vino,
    AUTHOR = {Vinogradov, I. M.},
     TITLE = {The method of trigonometrical sums in the theory of numbers},
      NOTE = {Translated from the Russian, revised and annotated by K. F.
              Roth and Anne Davenport,
              Reprint of the 1954 translation},
 PUBLISHER = {Dover Publications, Inc., Mineola, NY},
      YEAR = {2004},
     PAGES = {x+180},
      ISBN = {0-486-43878-3},
   MRCLASS = {11Lxx (01A75)},
  MRNUMBER = {2104806},
}

@article {WooNEC,
    AUTHOR = {Wooley, T. D.},
     TITLE = {Nested efficient congruencing and relatives of {V}inogradov's
              mean value theorem},
   JOURNAL = {Proc. Lond. Math. Soc. (3)},
  FJOURNAL = {Proceedings of the London Mathematical Society. Third Series},
    VOLUME = {118},
      YEAR = {2019},
    NUMBER = {4},
     PAGES = {942--1016},
      ISSN = {0024-6115,1460-244X},
   MRCLASS = {11L15 (11L05 11P55)},
  MRNUMBER = {3938716},
MRREVIEWER = {Moubariz\ Z.\ Garaev},
       DOI = {10.1112/plms.12204},
       URL = {https://doi.org/10.1112/plms.12204},
}

@book {DavBook,
    AUTHOR = {Davenport, H.},
     TITLE = {Analytic methods for {D}iophantine equations and {D}iophantine
              inequalities},
    SERIES = {Cambridge Mathematical Library},
   EDITION = {Second},
      NOTE = {With a foreword by R. C. Vaughan, D. R. Heath-Brown and D. E.
              Freeman,
              Edited and prepared for publication by T. D. Browning},
 PUBLISHER = {Cambridge University Press, Cambridge},
      YEAR = {2005},
     PAGES = {xx+140},
      ISBN = {0-521-60583-0},
   MRCLASS = {11P05 (11D72 11P55)},
  MRNUMBER = {2152164},
       DOI = {10.1017/CBO9780511542893},
       URL = {https://doi.org/10.1017/CBO9780511542893},
}

@article {BDG,
    AUTHOR = {Bourgain, J. and Demeter, C. and Guth, L.},
     TITLE = {Proof of the main conjecture in {V}inogradov's mean value
              theorem for degrees higher than three},
   JOURNAL = {Ann. of Math. (2)},
  FJOURNAL = {Annals of Mathematics. Second Series},
    VOLUME = {184},
      YEAR = {2016},
    NUMBER = {2},
     PAGES = {633--682},
      ISSN = {0003-486X,1939-8980},
   MRCLASS = {11P05 (11N25)},
  MRNUMBER = {3548534},
MRREVIEWER = {Ben\ Joseph\ Green},
       DOI = {10.4007/annals.2016.184.2.7},
       URL = {https://doi-org.ezproxy.lib.purdue.edu/10.4007/annals.2016.184.2.7},
}

@book {Stein,
    AUTHOR = {Stein, E. M.},
     TITLE = {Harmonic analysis: real-variable methods, orthogonality, and
              oscillatory integrals},
    SERIES = {Princeton Mathematical Series},
    VOLUME = {43},
      NOTE = {With the assistance of Timothy S. Murphy,
              Monographs in Harmonic Analysis, III},
 PUBLISHER = {Princeton University Press, Princeton, NJ},
      YEAR = {1993},
     PAGES = {xiv+695},
      ISBN = {0-691-03216-5},
   MRCLASS = {42-02 (35Sxx 43-02 47G30)},
  MRNUMBER = {1232192},
MRREVIEWER = {Michael\ Cowling},
}

@article {VDC,
    AUTHOR = {Van der Corput, J. G.},
     TITLE = {Zur {M}ethode der station\"aren {P}hase. {E}rste {M}itteilung
              {E}infache {I}ntegrale},
   JOURNAL = {Compositio Math.},
  FJOURNAL = {Compositio Mathematica},
    VOLUME = {1},
      YEAR = {1935},
     PAGES = {15--38},
      ISSN = {0010-437X,1570-5846},
   MRCLASS = {99-04},
  MRNUMBER = {1556874},
       URL = {http://www.numdam.org/item?id=CM_1935__1__15_0},
}

@book {Apo,
    AUTHOR = {Apostol, T. M.},
     TITLE = {Mathematical analysis},
   EDITION = {Second},
 PUBLISHER = {Addison-Wesley Publishing Co., Reading, Mass.-London-Don
              Mills, Ont.},
      YEAR = {1974},
     PAGES = {xvii+492},
   MRCLASS = {26-01 (28-01)},
  MRNUMBER = {344384},
}

@article {WW,
    AUTHOR = {Wei, B. and Wooley, T. D.},
     TITLE = {On sums of powers of almost equal primes},
   JOURNAL = {Proc. Lond. Math. Soc. (3)},
  FJOURNAL = {Proceedings of the London Mathematical Society. Third Series},
    VOLUME = {111},
      YEAR = {2015},
    NUMBER = {5},
     PAGES = {1130--1162},
      ISSN = {0024-6115,1460-244X},
   MRCLASS = {11L07 (11P05 11P32 11P55)},
  MRNUMBER = {3477231},
MRREVIEWER = {Xiumin\ Ren},
       DOI = {10.1112/plms/pdv048},
       URL = {https://doi-org.ezproxy.lib.purdue.edu/10.1112/plms/pdv048},
}

@article {Biggs,
    AUTHOR = {Biggs, K. D.},
     TITLE = {Almost equal summands in {W}aring's problem with shifts},
   JOURNAL = {Monatsh. Math.},
  FJOURNAL = {Monatshefte f\"ur Mathematik},
    VOLUME = {188},
      YEAR = {2019},
    NUMBER = {1},
     PAGES = {31--35},
      ISSN = {0026-9255,1436-5081},
   MRCLASS = {11D75 (11P05)},
  MRNUMBER = {3895390},
MRREVIEWER = {Yuchao\ Wang},
       DOI = {10.1007/s00605-018-1178-7},
       URL = {https://doi-org.ezproxy.lib.purdue.edu/10.1007/s00605-018-1178-7},
}

@phdthesis{Poulthesis,
  title        = {On {D}iophantine problems involving fractional powers of integers},
  author       = {Poulias, K.},
  year         = 2021,
  month        = {},
  note         = {Available at \url{https://research-information.bris.ac.uk/en/studentTheses/on-diophantine-problems-involving-fractional-powers-of-integers/}},
  school       = {University of Bristol},
  type         = {PhD thesis}
}

@phdthesis{Kuf,
  title        = {A real effective version of the Fre\u{i}man-Scourfield Theorem},
  author       = {K\"{u}fner, T.S.},
  year         = 2025,
  month        = {},
  note         = {Available at \url{http://dx.doi.org/10.53846/goediss-11614}},
  school       = {Georg-August University School of Science, G\"{o}ttingen},
  type         = {PhD thesis}
}

@article {Sal,
    AUTHOR = {Salmensuu, J.},
     TITLE = {On the {W}aring-{G}oldbach problem with almost equal summands},
   JOURNAL = {Mathematika},
  FJOURNAL = {Mathematika. A Journal of Pure and Applied Mathematics},
    VOLUME = {66},
      YEAR = {2020},
    NUMBER = {2},
     PAGES = {255--296},
      ISSN = {0025-5793,2041-7942},
   MRCLASS = {11P32 (11B30 11P05 11P55)},
  MRNUMBER = {4130325},
MRREVIEWER = {Bin\ Wei},
       DOI = {10.1112/mtk.12019},
       URL = {https://doi-org.ezproxy.lib.purdue.edu/10.1112/mtk.12019},
}

\noindent\textsc{Department of Mathematics, Purdue University, West Lafayette, IN, USA.}
\vspace{.03in}
\newline\noindent\textit{Email address}: bharga37@purdue.edu.
\end{document}